%% file: numLCE.tex
\documentclass[11pt]{article}
\usepackage{longtable}
\usepackage{mathrsfs}
\usepackage{multirow}
\usepackage{bigstrut}
\usepackage{amsmath,amsthm,amssymb, epsf, amsfonts}
\usepackage{graphicx}
\usepackage{color}
\usepackage{cases}

\normalbaselineskip
\input{my_definitions}

\begin{document}
\author{Chong Luo and Maria-Carme Calderer \\
School of Mathematics\\ University of Minnesota\\
Minneapolis, MN 55455\\
% { }\\
% and\\
% { }\\
% Satish Kumar
% \\
% Department of Chemical Engineering and Materials Sciences\\ University of Minnesota\\
% Minneapolis, MN 55455
}

% \begin{verse}
\title{Numerical Study of Liquid Crystal Elastomer Using Mixed Finite Element Method}
\bibliographystyle{siam}
\maketitle

\abstract{We aimed to use finite element method to simulate the unique behaviors of liquid crystal elastomer, 
such as semi-soft elasticity, stripe domain instabilities etc. We started from an energy functional with the 2D Bladon-Warner-Terentjev stored
energy of elastomer, the Oseen-Frank energy of liquid crystals, plus the penalty terms for the incompressibility constraint on 
the displacement, and the unity constraint on the director. 
Then we applied variational principles to get the differential equations.
Next we used mixed finite element method to do the numerical simulation. The existence, uniqueness, well-posedness and convergence of the 
numerical methods were investigated. The semi-soft elasticity was observed, and can be related to the rotation of the directors.
The stripe domain phenomenon, however, wasn't observed. This might due to the relative coarse mesh we have used.
}

\section{Introduction}
Nematic liquid crystal elastomer is a relatively new kind of elastic material 
which combines properties from both incompressible elasticity 
and rod-like liquid crystals. It has some unique behaviors such as the stripe-domain 
phenomenon and the semi-soft elasticity. These phenomena are usually
observed in the ``clamped-pulling'' experiement, in which
a piece of rectangular shape liquid crystal elastomer
is clamped at two ends and pulled along the direction
perpendicular to the initial alignment of the directors. 

Mitchell et al. \cite{mitchell1993strain} did the ``clamped-pulling''
experiement for \textit{acrylate-based} monodomain networks.
They found that the directors rotate in a discontinuous way: 
after a critical strain, the directors switch to the direction 
perpendicular to the original alignment.
On the other hand, Kundler and Finkelmann \cite{kundler1995strain} did 
the same experiment for \textit{polysiloxane}
liquid crystal elastomer. They found that
the directors rotate in a continuous manner.
They also observed the interesting stripe domain phenomenon.
Later, Zubarev and Finkelmann \cite{zubarev1999monodomain} re-investigated the 
``clamped-pulling'' experiment for polysiloxane liquid crystal elastomers,
to investigate the role of aspect ratio of the rectangular shape
monodomain in the formation of stripe domains. 
Figure \ref{fig:stripe-domain} shows the process of the ``clamped-pulling'' experiement 
in the case that the aspect ratio is $\mathrm{AR}=1$. 
In this graph, the pulling direction is vertical, while the initial alignment of
the directors is horizontal. The size of the elastomer was $5\times 5 mm^2$, with thickness
$50-300\pm 5$ $\mu m$.
 \begin{figure}[htbp]
 \centering
 \includegraphics[width=12cm]{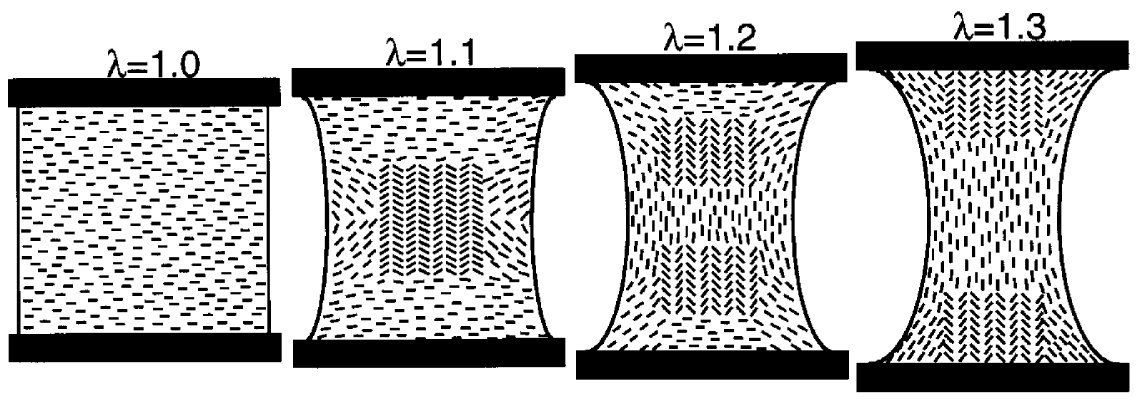}
 \caption{The stripe domain phenomenon. Reproduced from \cite{zubarev1999monodomain}}  \label{fig:stripe-domain}
 \end{figure}
They found that at some critical elongation factor $\lambda=1.1$, an opaque
region started to emerge in the center of the elastomer.
Using X-rays, they found that this region
was made of stripes of about $15 \mu m$ in width.
In each stripe, the directors were aligned uniformly,
while across the stripes, the directors were aligned
in a zig-zag way,
as shown in the second picture of Figure \ref{fig:stripe-domain}.
When the elastomer was pulled further, the opaque region
was broken into two regions symmetric about the center
(the third picture of Figure \ref{fig:stripe-domain}).
And these two regions moved closer to the two ends 
as the elastomer was pulled further. 
Eventually at $\lambda=1.3$ the two stripe domain regions reached the two ends 
(the last picture of Figure \ref{fig:stripe-domain})
and in this last stage, most of the directors at the center
had rotated 90 degrees, and were now aligned 
in the pulling direction (vertical direction).

Another interesting phenomenon of liquid crystal elastomers is
the so-called ``semi-soft elasticity''. 
 \begin{figure}[htbp]
 \centering
 \includegraphics[width=12cm]{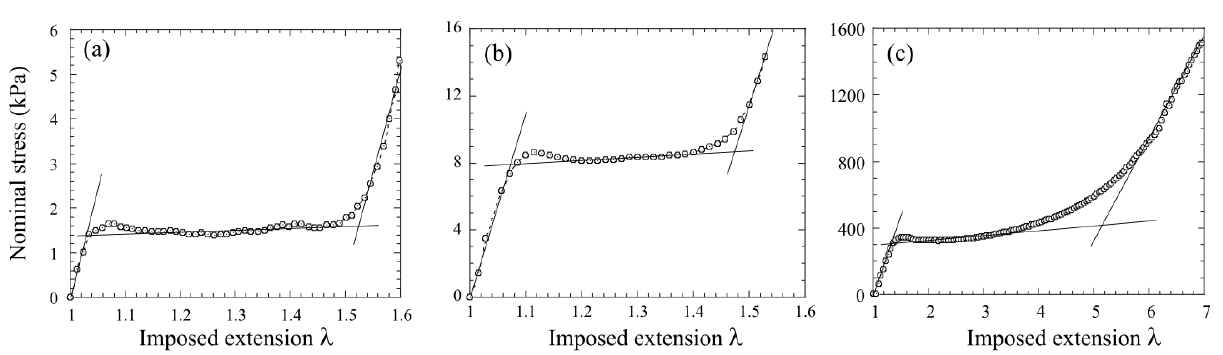}
 \caption{The semi-softness of liquid crystal elastomer. Reproduced from \cite{warner2007liquid}.}  \label{fig:semi-soft}
 \end{figure}
Figure \ref{fig:semi-soft} shows the stress-strain graph
of three different kinds of liquid crystal elastomers in
the ``clamped-pulling'' experiment (\cite{küpfer1994liquid, clarke2001effect}).
We can see that at the beginning, the nominal stress increases linearly
with the strain, just like ordinary elastic materials.
However, after the strain reaches a critical point,
the stress-strain curve becomes relatively ``flat'',
which we call the ``plateau region''. 
When the strain falls in this region, relatively small
stress can induce relatively large deformation, 
which means the elastomer is ``soft''. 
We can also see from Figure \ref{fig:semi-soft} that 
when the strain continues to increase and reaches another critical value,
the stress-strain curve becomes linear again.

The stored energy of liquid crystal elastomer proposed by Bladon, Warner and Terentjev (BTW) is \cite{bladon1993transitions}
\begin{equation} \label{eqn:BTW3D}
 W(F, \bn) = \mu \left(|F|^2-(1-a)|F^T\bn|^2-3a^{1/3}\right) 
\end{equation}
where $F=I+\nabla\bu(\boldsymbol{X})$ is the deformation gradient of the elastomer network, 
while $\bn=\bn(\boldsymbol{X})$ is a unit vector denoting the orientation
of the rod-like liquid crystal polymers, and $\boldsymbol{X}$ denotes an arbitrary point 
in the reference configuration. We assume that the elastomer is incompressible,
which is equivalent to $\mathrm{det}(F)=1$.
Also $\mu$ is the elasticity constant, and $a\in (0,1)$ is a constant measuring the
significance of the interaction between the displacement $\bu$ and
the orientation $\bn$. Notice that in the limit $a=1$, the BTW energy degenerates
to the stored energy of incompressible neo-Hookean materials. The closer 
the value $a$ is to $1$, the smaller the interaction between 
the displacement $\bu$ and the orientation $\bn$.
Similarly in 2D, the BTW energy can be defined as
\begin{equation} \label{eqn:BTW2D}
 W(F, \bn) = \mu \left(|F|^2-(1-a)|F^T\bn|^2-2a^{1/2}\right) 
\end{equation}
The only difference is that the constant term is changed from $3a^{1/3}$
to $2a^{1/2}$.
\begin{proposition} \label{prop:BTW2Dzero}
 The BTW energy $W(F, \bn)$ in equation (\ref{eqn:BTW3D}) and (\ref{eqn:BTW2D}) are always non-negative.
 \begin{itemize}
  \item In 3D, $W(F, \bn)=0$
 if and only if $\mathrm{eig}(F F^T)=\{a^{1/3},a^{1/3}, a^{-2/3}\}$ and $\bn$
 is an eigenvector of the eigenvalue $ a^{-2/3}$.
 \item In 2D, $W(F, \bn)=0$
 if and only if $\mathrm{eig}(F F^T)=\{a^{1/2}, a^{-1/2}\}$ and $\bn$
 is an eigenvector of the eigenvalue $ a^{-1/2}$.
 \end{itemize}
\end{proposition}

\begin{proposition} \label{prop:ShearEnergy2D}
 \begin{itemize}
  \item In 3D, assume $\lambda(t)=\lambda(t)=a^{1/6}(1-t)+a^{-1/3}t$ with $0\leq t \leq 1$. When there is no shearing and
\begin{equation} \label{eqn:uniformF}
 F = 
 \begin{pmatrix}
  a^{1/6}&    0&               0\\
  0&          \lambda(t)&      0\\
  0&          0&              a^{-1/6}/\lambda(t)
 \end{pmatrix}
\end{equation}
 the BTW energy (\ref{eqn:BTW3D}) is zero for $t=0,1$ and strictly positive for $0<t<1$. However, when there is shearing and
 \begin{equation} \label{eqn:shearingF}
 F = 
 \begin{pmatrix}
  a^{1/6}&    0&               0\\
  0&          \lambda(t)&      \pm\delta \\
  0&          0&              a^{-1/6}/\lambda(t)
 \end{pmatrix}
\end{equation} 
 with
 \begin{equation} \label{eqn:shearingdelta}
  \delta = \sqrt{a^{1/3}+a^{-2/3}-(\lambda^2+a^{-1/3}/\lambda^2)},
 \end{equation}
 the BTW energy (\ref{eqn:BTW3D}) can be zero for $0\leq t \leq 1$.
  \item In 2D, assume $\lambda(t)=a^{1/4}(1-t)+a^{-1/4}t$ with $0\leq t \leq 1$. When there is no shearing and
 \begin{equation}
  F = 
  \begin{pmatrix}
   \lambda(t)&   0 \\
   0&            1/\lambda(t)
  \end{pmatrix}
 \end{equation}
 the BTW energy (\ref{eqn:BTW2D})  is zero for $t=0,1$ and  is strictly positive for $0<t<1$. However, when there is shearing and
 \begin{equation}
  F = 
  \begin{pmatrix}
   \lambda(t)&   \pm \delta \\
   0&            1/\lambda(t)
  \end{pmatrix}
 \end{equation}
 with
 \begin{equation}
  \delta = \sqrt{a^{1/2}+a^{-1/2}-(\lambda^2+\lambda^{-2})},
 \end{equation}
 the BTW energy (\ref{eqn:BTW2D}) can be zero for $0\leq t \leq 1$.
 \end{itemize}
\end{proposition}
\begin{corollary} \label{cor:nonconvexBTW}
 If we take 
 \begin{align}
  W(F) ={}& \min_{|\bn|=1} W_{BTW}(n, F), \\
       ={}&
       \begin{cases}
          \mu \left(\lambda_1^2+\lambda_2^2+a\lambda_3^2-3a^{1/3}\right)& \text{ in 3D} \\
          \mu \left(\lambda_1^2+a\lambda_2^2-2a^{1/2}\right)& \text{ in 2D} \\
       \end{cases}
 \end{align}
 then $W(F)$ is a non-convex function of $F$. 
\end{corollary}
\begin{proof}
  Take $0<t<1$, and $\lambda(t)$ and $F_{\pm}$ as defined Proposition \ref{prop:ShearEnergy2D}. Then we have
 \begin{eqnarray*}
  W(0.5F_{+} + 0.5F_{-}) > 0
 \end{eqnarray*}
 while
 \begin{eqnarray*}
  0.5W(F_{+}) + 0.5W(F_{-}) = 0
 \end{eqnarray*}
 Therefore $W(F)$ is a non-convex function of $F$.
\end{proof}

Proposition \ref{prop:ShearEnergy2D} tells us that introducing shearing 
might lower the BTW energy (\ref{eqn:BTW3D}) or (\ref{eqn:BTW2D}). 
On the other hand, because of the constraint of the clamps,
global shearing is not possible, therefore, local shearing
is developed in a zig-zag way, and that's why stripe domains occur.

In \cite{conti2002soft}, DeSimone et al. did the finite element 
study of the clamped pulling of liquid crystal elastomer. 
They started from the 3D BTW energy (\ref{eqn:BTW3D}),
and got rid of the variable $\bn$ by taking it 
to be always $\bn=\bn(F)$, the eigenvector of the largest eigenvalue $a^{-2/3}$
of the matrix $F F^T$. So they got
\begin{equation} \label{eqn:degBTW3D}
 W(F) = \mu\left(\lambda_1^2(F)+\lambda_2^2(F)+a\lambda_3^2(F)-3a^{1/3}\right)
\end{equation}
It turned out that the energy function (\ref{eqn:degBTW3D}) is no longer a convex function of $F$
(Corollary \ref{cor:nonconvexBTW}).
Therefore, they replaced it by its quasi-convex envelope
\begin{align}
 W_{qc}(F): = \inf_{\stackrel{y \in W^{1,\infty}}{\mathrm{det}(\nabla y(x))=1}} 
 \left\{ \frac{1}{|\Omega|} \int_{\Omega} W(\nabla y(x))dx: \quad
   y(x)=Fx \text{ on }\partial\Omega  \right\}
\end{align}
They pointed out, that the use of $W_{qc}(F)$ in the numerical computations allows one
to resolve only the \textit{macroscopic} length scale, with the (possibly infinitesimal)
microscopic scale already accounted for in $W_{qc}(F)$. 

It turned out that their approach was very successful, in that they successfully captured
the emerging and migrating of the stripe domain regions, in exactly the same way
as those observed experimentally in \cite{zubarev1999monodomain}.
However, their approach is limited for several reasons:
\begin{itemize}
 \item Although they were able to tell which part of the region is a ``stripe-domain region'',
   they couldn't tell how many stripes lie in that region. Actually, their model
   permits infinitely many stripes in that region, which is unphysical.
 \item They were only able to observe the ``soft elasticity'', not the ``semi-soft elasticity''.
 That is, the plateau region occurs immediately upon the pulling (Figure \ref{fig:desimone-ss}), 
 instead of after some critical strain
 as observed experimentally. 
\end{itemize}
 \begin{figure}[htbp]
 \centering
 \includegraphics[width=8cm]{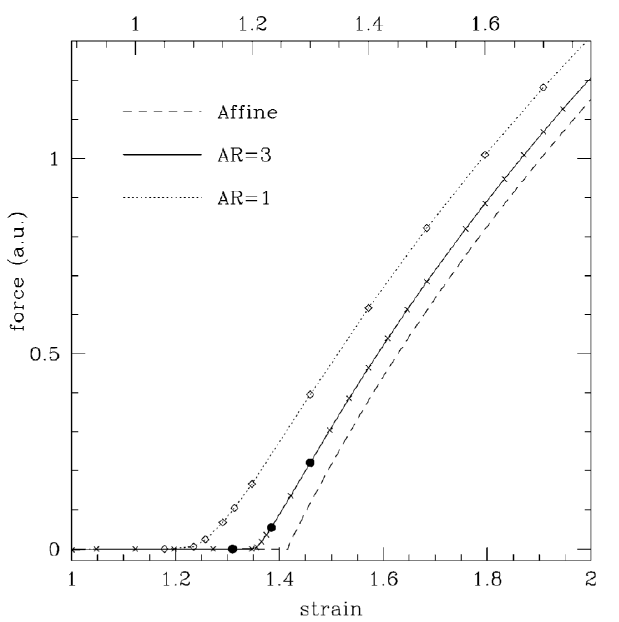}
 \caption{The stress-strain curve from Desimone's numerical simulation. 
 Soft elasticity instead of semi-soft elasticity was observed. Reproduced from \cite{conti2002soft}.}  \label{fig:desimone-ss}
 \end{figure}

In this work, we aim to resolve the above issues. 
We combine the 2D BTW model (\ref{eqn:BTW2D}) and the Oseen-Frank model (\cite{frank1958liquid}), 
and get the following energy functional for liquid crystal elastomers:
\begin{align*}
 \Pi(\bu,\bn) ={} & \int_{\Omega}c \left(|F|^2-(1-a)|F^T \bn|^2 \right)+b|\nabla \bn|^2 \nonumber \\
              &-\int_{\Omega} \f\cdot \bu -\int_{\Gamma_2} \bg\cdot \bu da
\end{align*}
where $\f$ is the body force, while $\bg$ is the traction force on the $\Gamma_2$
part of the boundary $\partial \Omega$.
We assume $\bu \in H^1(\Omega)$, and satisfies $\det(I+\nabla \bu)=1$ a.e. in $\Omega$;
while $\bn \in H^1(\Omega)$, and satisfies $|\bn|=1$ a.e. in $\Omega$.
Let $L$ be the characteristic length scale of the domain $\Omega$, we get the following
non-dimensionalized energy
\begin{align*}
 \tilde{\Pi}(\bu,\bn) = {} & \int_{\tilde{\Omega}} (|F|^2-(1-a)|F^T \bn|^2)+\tilde{b}|\nabla \bn|^2 \nonumber \\
              &-\int_{\tilde{\Omega}} \tilde{\f}\cdot \bu -\int_{\tilde{\Gamma}_2} \tilde{\bg}\cdot \bu da
\end{align*}
where $\tilde{\Omega}$ is congruent with $\Omega$ with characteristic length scale $1$, 
and $\tilde{\Gamma}_2 \subset \partial\tilde{\Omega}$ 
is the image of $\Gamma_2$.
Also $\tilde{b}=\frac{b}{c L^2}$, $\tilde\f=\frac\f{c L}$, and $\tilde\bg = \frac\bg{c}$.

The rest of the paper is organized as the following. In section \ref{sec:continous-problem},
we investigated the properties of the continuous problem, 
such as the existence of minimizer, the equilibrium equation
and its linearization, the stress-free state, 
the well-posedness of the linearized equation etc.
In section \ref{sec:discretization}, we investigated
problems related to the discretization, 
such as the existence of minimizer, 
the equilibrium equation and its linearization, 
the well-posedness of the linearized equation,
and the existence and uniqueness of
the Lagrange multipliers etc. 
In section \ref{sec:numerical-results},
we presented the simulation results
for the ``clamped-pulling'' experiment,
including the results of the inf-sup tests,
the convergence rates, and the stress-strain curve etc.
Finally, in section \ref{sec:discussion},
we discussed future directions.

\section{The Continuous Problem}\label{sec:continous-problem}
\subsection{Existence of minimizer}
Let 
\begin{equation}
 H^1(\Omega, S^1) = \{\bn\in H^1(\Omega, \mathbb{R}^2): |\bn|=1 \text{ almost everywhere in } \Omega\}
\end{equation}
and
\begin{equation}
 K = \{\bu\in H^1(\Omega, \mathbb{R}^2): \det (I+\nabla \bu)=1 \text{ almost everywhere in } \Omega\}
\end{equation}
Define the admissible set 
\begin{equation}
 \mathcal{A}(\bu_0, \bn_0) = \{\bu\in K, \bn\in H^1(\Omega, S^1): 
 \bu=\bu_0 \text{ on }\Gamma_0, \bn=\bn_0 \text{ on }\Gamma_1\}
\end{equation}
Define the energy functional
\begin{align}
 \Pi(\bu,\bn) ={}& \int_{\Omega} \left(|F|^2-(1-a)|F^T \bn|^2\right)+b|\nabla \bn|^2 \nonumber \\
              &-\int_{\Omega} \f\cdot \bu -\int_{\Gamma_2} \bg\cdot \bu da
\end{align}
where $F=I+\nabla \bu$.
Then our problem is
\begin{align} \label{problem:minBTWOF}
 \text{Find } (\bu,\bn)\in \mathcal{A}(\bu_0, \bn_0) \text{ minimizing } \Pi(\bu,\bn).
\end{align}

\begin{lemma}\label{lemma:coerBTW}
 Assume $|\bn|=1$, we have
 \begin{equation}
  \left(|F|^2-(1-a)|F^T \bn|^2\right) \geq a|F|^2
 \end{equation}
\end{lemma}
\begin{proof}
 Let $\lambda_1^2 \leq \lambda_2^2$ be the two eigenvalues of the real symmetric matrix $F F^T$.
 Since $\bn$ is in $H^1(\Omega, S^1)$, by Proposition \ref{prop:BTW2Dzero}, we have
 \begin{align*}
  \left(|F|^2-(1-a)|F\bn|^2\right) 
  \geq{}& \lambda_1^2 + a\lambda_2^2 \\
  \geq{}& a(\lambda_1^2+\lambda_2^2) \\
  ={}& a\tr (F F^T) \\
  ={}& a|F|^2 
 \end{align*}
\end{proof}

\begin{lemma} \label{lemma:BTWconvex}
 Assume $|\bn|=1$, and $0<a<1$. The function
 \begin{equation}
  L(F, \bn) = |F|^2-(1-a)|F^T \bn|^2
 \end{equation}
 is convex in $F$.
\end{lemma}
\begin{proof}
 Let 
 \begin{equation}
  A(\bn) = I-(1-a)\bn\bn^T
 \end{equation}
 Then we have
 \begin{equation}
  L = \tr (FAF^T)
 \end{equation}
 Then for any $F_1, F_2 \in \mathbb{M}^{3\times3}$, and $0\leq \alpha \leq 1$, we have
 \begin{align*}
  &[\alpha L(F_1)+(1-\alpha)L(F_2)]-L(\alpha F_1+(1-\alpha)F_2) \\
  ={}& \alpha\tr (F_1 A F_1^T) + (1-\alpha)\tr (F_2 A F_2^T) \\
    & - \tr [(\alpha F_1+(1-\alpha)F_2)A(\alpha F_1+(1-\alpha)F_2)^T] \\
  ={}& \alpha(1-\alpha)\tr [(F_1-F_2)A(F_1-F_2)] \\
  ={}& \alpha(1-\alpha)\left(|F_1-F_2|^2-(1-a)|(F_1-F_2)^T \bn|^2\right) \\
  \geq{}& 2\alpha(1-\alpha)\sqrt{a} \\
  \geq{}& 0
 \end{align*}
 where we have used Proposition \ref{prop:BTW2Dzero}.
 Therefore $L$ is a convex function of $F$.
\end{proof}

\begin{remark}
 It's easy to see that Lemma \ref{lemma:coerBTW} and Lemma \ref{lemma:BTWconvex}
 holds for 3D, too.
\end{remark}

\begin{theorem} \label{thm:detconv}
 Let $\Omega$ be a nonempty, bounded, open subset of $\mathbb{R}^d$. 
 \begin{itemize}
  \item If $d=2$, suppose we have $\bu_k \rightharpoonup \bu$ in $W^{1,s}$ with $s>\frac{4}{3}$, 
 then we have $\mathrm{det}(I+\nabla \bu_k) \rightarrow \mathrm{det}(I+\nabla \bu)$ in $\mathcal{D}'(\Omega)$. 
  \item If $d=3$, 
 \begin{itemize}
  \item suppose we have $\bu_k \rightharpoonup \bu$ in $W^{1,s}$ with $s>\frac{3}{2}$, 
  then we have $\mathrm{adj}(I+\nabla \bu_k)_{ij} \rightarrow \mathrm{adj}(I+\nabla \bu)_{ij}$ in $\mathcal{D}'(\Omega)$.
  \item suppose we have $\bu_k \rightharpoonup \bu$ in $W^{1,s}$, 
 and $\mathrm{adj}(I+\nabla \bu_k) \rightharpoonup \mathrm{adj}(I+\nabla \bu)$ in $L^q(\Omega;\mathbb{M}^3)$
 with $s>1, q>1$ and $\frac{1}{s}+\frac{1}{q}<\frac{4}{3}$,  
 then we have $\mathrm{det}(I+\nabla \bu_k) \rightarrow \mathrm{det}(I+\nabla \bu)$ in  $\mathcal{D}'(\Omega)$. 
 \end{itemize}
 \end{itemize}
\end{theorem}
For the proof, please see John Ball \cite{ball1976convexity}.

\begin{theorem} \label{thm:cont-exist-BTWOF}
 There exists solution to the problem (\ref{problem:minBTWOF}).
\end{theorem}
\begin{proof}
 Let $m$ be the infimum of $\Pi$ on $\mathcal{A}(\bu_0, \bn_0)$, 
 and let $(\bu_k,\bn_k) \in \mathcal{A}(\bu_0, \bn_0)$ be a minimizing sequence of $\Pi$.
 Obviously $m<+\infty$.
 Thus $\Pi(\bu_k, \bn_k)$ is bounded above.
 By Lemma \ref{lemma:coerBTW}, we have
 \begin{align*} 
  \Pi(\bu_k, \bn_k) 
  \geq{}& \int_{\Omega} a|(I+\nabla \bu_k)|^2+b|\nabla \bn_k|^2 dx \\
       &-\|\f\|_{L^2(\Omega)} \|\bu_k\|_{L^2(\Omega)}-\|\bg\|_{L^2(\Gamma_2)} \|\bu_k\|_{L^2(\Gamma_2)} \\
    \geq{}&  \int_{\Omega} a|(I+\nabla \bu_k)|^2 + b|\nabla \bn_k|^2 dx \\
     & -\left(\frac{1}{\varepsilon}\|\f\|_{L^2(\Omega)}^2 + \varepsilon\|\bu_k\|_{L^2(\Omega)}^2 \right)
     -\left(\frac{1}{\varepsilon}\|\bg\|_{L^2(\Gamma_2)}^2 + \varepsilon\|\bu_k\|_{L^2(\Gamma_2)}^2 \right) \\
   \geq{}&  \int_{\Omega} C_1 |(I+\nabla \bu_k)|^2 + C_2 |\nabla \bn_k|^2 dx - C_3 
 \end{align*} 
  where $\varepsilon >0$ is small and $C_i>0, i=1,2,3$ are constants and 
  we have applied the generalized Poincar$\mathrm{\acute{e}}$ inequality (\cite{CI87}, p281) 
  and the Trace Theorem (\cite{evans1998c}, p258) in the last step.
 Therefore we have $F_k=I+\nabla \bu_k$ and $\nabla \bn_k$ are bounded in $L^2$.
 Now we have that $\nabla(\bu_k-\bu_0)$ is bounded in $L^2$, by the Poincar$\mathrm{\acute{e}}$ inequality, we have
 $\bu_k$ is bounded in $H^1$. On the other hand, since $\nabla \bn_k$ is bounded in $L^2$ and $\bn_k$ 
 is in $H^1(\Omega, S^1)$, we can get that $\bn_k$ is bounded in $H^1$.
 Now since $H^1$ is a reflexive Banach space and $\bu_k$ and $\bn_k$ are bounded in $H^1$, 
 we can find a subsequence of $\bu_k$ and a subsequence of $\bn_k$
 such that they are weakly convergent in $H^1$. We still denote them as $(\bu_k, \bn_k)$,
 and assume $\bu_k \rightharpoonup \bu$, $\bn_k \rightharpoonup \bn$.
 
 Now that $\bu_k \rightharpoonup u$ in $H^1$, by Theorem \ref{thm:detconv}, we have
 $\det (I+\nabla \bu_k) \rightarrow \det (I+\nabla \bu)$ in $\mathcal{D}'(\Omega)$. 
 Since $\det (I+\nabla \bu_k)=1$ a.e., we have $\det (I+\nabla \bu)=1$ a.e. in $\Omega$
 \footnote{This is because, by definition, we have 
 \begin{equation*}
 \langle\det (I+\nabla \bu_k), \phi\rangle \rightarrow \langle\det (I+\nabla \bu), \phi\rangle
 \end{equation*} 
 in $\mathbb{R}$ for any $\phi\in \mathcal{D}(\Omega)$. Since $\det (I+\nabla \bu_k)=1$ a.e. in $\Omega$ for any $k$, 
 we have
 \begin{equation*}
  \langle\det (I+\nabla \bu)-1, \phi\rangle = 0, \quad \forall \phi \in \mathcal{D}(\Omega)
 \end{equation*}
 Thus $\det (I+\nabla \bu)=1$ a.e. in $\Omega$. 
 }.
 On the other hand, weak convergence in $H^1$ implies strong convergence in $L^2$
 \footnote{This is because the embedding $I: W^{1,p}\rightarrow L^p$ is compact for $1\leq p \leq \infty$ (\cite{evans1998c}, p274),
 while for any compact operator $A: V \rightarrow W$ with $V$ and $W$ Banach spaces, $u_k \rightharpoonup u$ in $V$ implies 
 $A u_k \rightarrow Au$ in $W$ (\cite{CI87}, Theorem 7.1-5 on p348).},
 thus we can find a subsequence of $\bn_k$ that converges point-wise almost everywhere.
 Therefore we have $|\bn|=1$ almost everywhere.  
 Finally since $\bu_k-\bu_0 \in H^1_{0|\Gamma_0}$, and $H^1_{0|\Gamma_0}$ is a closed linear
 subspace of $H^1$, by the Mazur's Theorem, it's weakly closed. 
 Therefore $\bu-\bu_0$ is also in $H^1_{0|\Gamma_0}$, which means $\bu=\bu_0$ on $\Gamma_0$.
 Similarly we have $\bn=\bn_0$ on $\Gamma_1$. Therefore we have $(\bu,\bn)\in \mathcal{A}(\bu_0,\bn_0)$.

 By Lemma \ref{lemma:BTWconvex}, the following function
 \begin{align*}
  L(F, \bn, P) ={}& \left(|F|^2-(1-a)|F^T \bn|^2\right) + b|P|^2 
 \end{align*}
 is a convex function of $F$ and $P$. Therefore by Theorem 1 of section 8.2 of \cite{evans1998c},
 $\Pi$ is weakly lower semi-continuous. 
 Thus we have
 \begin{align*}
  \Pi(\bu,\bn) \leq{}& \liminf_{k\rightarrow \infty} \Pi(\bu_k, \bn_k)\\
            ={}& m
 \end{align*}
 Since $m$ is the infimum of $\Pi$ on $\mathcal{A}$, we conclude that
 \begin{equation}
   \Pi(\bu,\bn) = m
 \end{equation}
 That is, $(\bu,\bn)$ is the minimizer of $\Pi$ on $\mathcal{A}$.
\end{proof}
\begin{remark}
 Theorem \ref{thm:detconv} is crucial in this existence proof. 
 Since in the 3D case, we need the additional condition that $\mathrm{adj}(F_k)$
 is weakly convergent to get the convergence of $\det (F_k)$,
 the above proof cannot be directly extended to 3D.
\end{remark}

\subsection{Equilibrium equation and its linearized system}
Constrained minimization problems can usually be reduced to unconstrained minimization problems
by introducing Lagrange multipliers. 
We consider the following non-dimensionalized energy functional
\begin{align} \label{nondimEnergy}
 \mathcal{E}(\bu,p,\bn,\lambda) ={}& \int_{\Omega} (|F|^2-(1-a)|F^T \bn|^2)
    +b|\nabla \bn |^2 \nonumber\\
     &-p(\det F-1)+\lambda(|\bn|^2-1)-\int_{\Omega}\f\cdot \bu-\int_{\Gamma_2} \bg\cdot \bu
\end{align}
where $p$ is the Lagrange multiplier for the incompressibility constraint $\det(I+\nabla\bu)=1$,
which can be interpreted as ``pressure'', while $\lambda$ is the Lagrange multiplier 
for the unity constraint $|\bn|=1$.

Assume $(\bu,\bn,p,\lambda)$ minimizes the non-dimensionalized energy (\ref{nondimEnergy}).
Then for any test function $\bv$,
the function $\mathcal{E}(\varepsilon) = \mathcal{E}(\bu+\varepsilon\bv,p,\bn,\lambda)$
has a minimum at $\varepsilon = 0$. Thus we have
\begin{equation}
  0 = \frac{d\mathcal{E}}{d\varepsilon} (0)
\end{equation}
from which we can get the following equation
\begin{align*}
 0 ={}&  \int_{\Omega} 2\left(F:\nabla \bv-(1-a)(F^T\bn, \nabla \bv^T \bn)\right) 
       -p\frac{\partial \det }{\partial F}:\nabla \bv \nonumber \\
       & -\int_{\Omega} \f\cdot \bv-\int_{\Gamma_2} \bg\cdot \bv da \\
\end{align*}

By similarly taking the variations $\bn\rightarrow \bn+\varepsilon \bm$,
$p\rightarrow p+\varepsilon q$, or $\lambda \rightarrow \lambda+\varepsilon \mu$,
we get the following set of Euler-Lagrange equations
\begin{align}
 0 ={}&  \int_{\Omega} 2\left(F:\nabla \bv-(1-a)(F^T\bn, \nabla \bv^T \bn)\right) 
       -p\frac{\partial \det }{\partial F}:\nabla \bv \nonumber \\
     & -\int_{\Omega} \f\cdot \bv-\int_{\Gamma_2} \bg\cdot \bv da \nonumber \\
0 ={}& \int_{\Omega} -2(1-a)(F^T \bn, F^T \bm) +2b(\nabla \bm, \nabla \bn)+2\lambda(\bn,\bm) \nonumber \\
0 ={}& \int_{\Omega}-q(\det F-1) \nonumber \\
0 ={}& \int_{\Omega}\mu ((\bn,\bn)-1) \nonumber
\end{align}
in which we seek solutions $(\bu,\bn,p,\lambda) \in {\mathbf H}^1_{\bg|\Gamma_0}
\times {\mathbf H}^1_{{\bg'}|\Gamma_1}\times L^2(\Omega)\times L^2(\Omega)$. Correspondingly, the test functions
$(\bv,\bm,q,\mu)$ are in ${\mathbf H}^1_{0|\Gamma_0}
\times {\mathbf H}^1_{0|\Gamma_1}\times L^2(\Omega)\times L^2(\Omega)$.
Here $H^1_{g|\Gamma_0}$ is defined as
\begin{align*}
 H^1_{g|\Gamma_0} ={}& \{v \in H^1(\Omega;{\mathbb R}), v|_{\Gamma_0} = {g}\}
\end{align*}
and ${\mathbf H}^1_{\bg|\Gamma_0}$ is the corresponding vector version.

We linearize the Euler-Lagrange equation around a solution  $(\bu,\bn,p,\lambda)$, and get the following 
linearized equations:
\begin{align}
 a_1(\bw, \bv)+a_2(\bl, \bv)+b_1(o, \bv) ={}& L_1(\bv)
      \qquad \forall \bv \in \mathbb{V} \label{mixed1}\\
 a_2(\bm, \bw)+a_3(\bl, \bm)+b_2(\gamma, \bm) ={}& L_2(\bm)
      \qquad \forall \bm \in \mathbb{M} \label{mixed2} \\
 b_1(q, \bw)  ={}& L_3(q) \qquad \forall q \in \mathbb{P} \label{mixed3}\\
 b_2(\mu, \bl) ={}& L_4(\mu) \qquad \forall \mu \in \Lambda\label{mixed4}
\end{align}
where $\mathbb{V} = {\mathbf H}^1_{0|\Gamma_0}$, $\mathbb{M} = {\mathbf H}^1_{0|\Gamma_1}$,
$\mathbb{P} = L^2(\Omega)$, $\Lambda = H^{-1}_{\Gamma_1}$,
and $(\bw,\bl,o,\gamma) \in \mathbb{V}\times\mathbb{M}\times\mathbb{P}\times\Lambda$ is the change in the solution,
$(\bv,\bm,q,\mu) \in \mathbb{V}\times\mathbb{M}\times\mathbb{P}\times\Lambda$ is a test function. Here
 $H^{-1}_{\Gamma_1}$ is the dual space of $H_{0|\Gamma_1}^1$. We'll abuse the notation and
use $\|\cdot\|_{-1}$ to also denote the norm
in the $H^{-1}_{\Gamma_1}$ space.
The bilinear forms in the linearized system are defined in the following equations
\begin{align}
 a_1(\bw,\bv)={}&\int_{\Omega} 2(\nabla \bw, \nabla \bv)
-2(1-a)(\nabla \bw^T \bn, \nabla \bv^T \bn) \nonumber \\
  &-p(\frac{\partial^2 \det }{\partial F^2} \nabla \bw):\nabla \bv
\end{align}

\begin{align}
 a_2(\bm,\bv)=&\int_{\Omega} -2(1-a)(F^T \bm, \nabla \bv^T \bn)
-2(1-a)(F^T \bn, \nabla \bv^T \bm) 
\end{align}

\begin{align}
 a_3(\bl,\bm)=&\int_{\Omega} -2(1-a)(F^T \bl, F^T \bm)
              +2b(\nabla \bm, \nabla \bl)
                         +2\lambda(\bl, \bm)
\end{align}

\begin{equation} \label{b1}
 b_1(q, \bw) = \int_{\Omega} -q \frac{\partial \det }{\partial F}: \nabla \bw
\end{equation}

\begin{equation} \label{b2}
 b_2(\mu, \bl) = \int_{\Omega} 2\mu(\bl, \bn)
\end{equation}

\subsection{The stress-free state} \label{sec:stress-free}
From the Euler-Lagrange equations, we can get the following strong form partial differential equations:
\begin{align}
 -\mathrm{div} \sigma ={}& \f  \qquad \text{ in } \Omega \label{strong1}\\
 b\mathrm{div}(\nabla \bn)+(1-a)\bn^T F F^T -\lambda \bn^T ={}& 0
  \qquad \text{ in } \Omega \label{strong2}\\
 \det F ={}& 1 \qquad \text{ in } \Omega \label{strong3} \\
 (\bn, \bn) ={}& 1 \qquad \text{ in } \Omega \label{strong4}
\end{align}
and the following natural boundary conditions:
\begin{align*}
 \sigma \vec{\nu} ={}& \bg       \qquad \text{ on } \partial\Omega\backslash\Gamma_0 \\
 \nabla \bn \vec{\nu} ={}& 0 \qquad \text{ on } \partial\Omega\backslash\Gamma_1
\end{align*}
where the Piola-Kirchhoff stress tensor is
\begin{equation}
 \sigma = 2\left(I-(1-a)\bn\bn^T\right)F-p\frac{\partial \det }{\partial F}
\end{equation}
Notice that if we choose the displacement to be $\bu\equiv 0$ and the director to be $\bn\equiv (0,1)^T$,
then we have
\begin{equation*}
 \sigma = 
 \begin{pmatrix}
  2-p& 0\\
  0&  2a-p\\
 \end{pmatrix}
\end{equation*}
If there is zero body force, equation (\ref{strong1}) implies that $p$ must be a constant. Then there is no way
to make this $\sigma$ zero if $a \neq 1$.
However, it can be verified that the stress-free state can be achieved with
  \begin{equation}
    F = 
    \begin{pmatrix}
     a^{1/4}&        0\\
     0&            a^{-1/4}
    \end{pmatrix}
  \end{equation} 
and
\begin{align*}
  \bn \equiv{}& (0,1)^T \\
  p ={}& 2\sqrt{a} \\
  \lambda ={}& (1-a)/\sqrt{a}
\end{align*}
\begin{remark}
 The above observation implies that the reference configuration (strain-free state) for the BTW model
 is different from the stress-free configuration. 
 It can be verified that if we take the reference state to be the stress-free state instead,
 then the BTW model becomes the so-called ``general BTW'' model
 \begin{equation}
   W_{BTW} = \frac{1}{2} \mu \tr (L_0 F^T L^{-1} F)
 \end{equation}
 where 
 \begin{equation*}
 L(\bn) = a^{-1/2} \bn \bn^T + a^{1/2}(I-\bn \bn^T)
\end{equation*}
 is the so-called \textit{step tensor}, and $L_0 = L(\bn_0)$ is the step tensor at the stress-free state.
\end{remark}

\subsection{Well-posedness of the linearized system}
By adding (\ref{mixed1}) and (\ref{mixed2}) together, and adding (\ref{mixed3})-(\ref{mixed4}) together,
we can reduce the linearized system into a standard saddle point system
\begin{align}
 a(\tilde\bw, \tilde\bv)+b(\tilde{o}, \tilde\bv) ={}& \tilde{L}_1(\tilde\bv)
 \qquad \forall \tilde\bv\in \tilde{\mathbb{V}}
  \label{1foldmixed1} \\
 b(\tilde{q}, \tilde\bw) ={}& \tilde{L}_2(\tilde{q}) \qquad \forall \tilde{q}\in\tilde{\mathbb{P}} \label{1foldmixed2}
\end{align}
where $\tilde{\mathbb{V}} = \mathbb{V} \times \mathbb{M}$, $\tilde{\mathbb{P}} = \mathbb{P} \times \Lambda$, and
\begin{align}
 a(\tilde\bw, \tilde\bv) ={}& a_1(\bw,\bv)+a_2(\bl,\bv)
                                           +a_2(\bm,\bv)+a_3(\bl,\bm) \label{asum}\\
 b(\tilde{q}, \tilde\bv) ={}& b_1(q, \bv)+b_2(\mu, \bm)
\end{align}

The well-posedness of saddle point systems are well established. Here we quote the
so-called Ladyzenskaya-Babuska-Brezzi theorem from \cite{LeTallec}.
\begin{theorem}[Ladyzenskaya-Babuska-Brezzi]\label{Thm-LBB}
Consider the following saddle point problem
\begin{align}
 &a(u,v)+b(p,v)=L_V(v) \qquad \forall v\in \mathbb{V}, u\in \mathbb{V} \label{1foldmixed1-orig} \\
 &b(q,u) = L_P(q) \qquad \forall q\in \mathbb{P}, p\in\mathbb{P} \label{1foldmixed2-orig}
\end{align}
with $\mathbb{V}$ and $\mathbb{P}$ given Hilbert spaces, $L_V$ and $L_P$ belonging to $\mathbb{V}'$
and $\mathbb{P}'$ respectively. Moreover, $a$ and $b$ are continuous bilinear forms defined
on $\mathbb{V}\times\mathbb{V}$ and $\mathbb{P}\times\mathbb{V}$, respectively. Define the operators
\begin{equation*}
 \begin{split}
  \mathscr{B}: & \mathbb{V} \rightarrow \mathbb{P}' \\
     & v \mapsto \mathscr{B}v
     \text{ such that } \langle \mathscr{B}v,q\rangle = b(q,v)
     \qquad \forall q\in \mathbb{P}
 \end{split}
\end{equation*}
\begin{equation*}
 \begin{split}
  \mathscr{A}: & \mathrm{Ker}\mathscr{B} \rightarrow (\mathrm{Ker}\mathscr{B})' \\
     & w \mapsto \mathscr{A}w
     \text{ such that } \langle \mathscr{A}w,v\rangle = a(w,v)
     \qquad \forall v\in \mathrm{Ker}\mathscr{B}
 \end{split}
\end{equation*}
Then the operator $\mathscr{B}$ is onto if and only if the spaces $\mathbb{V}$ and $\mathbb{P}$ satisfy the
following inf-sup condition:
\begin{equation} \label{infsupb-orig}
 \inf_{q\in\mathbb{P}, \|q\|=1} \sup_{{v}\in\mathbb{V}, \|{v}\|=1} b(q, {v}) \geq \beta >0
\end{equation}
Moreover, the mixed problem is well-posed if and only if $\mathscr{B}$ is onto and $\mathscr{A}$ is invertible.
\end{theorem}
\begin{remark}
 The operator $\mathscr{A}$ is invertible if and only if the following inf-sup condition is satisfied
\begin{equation} \label{infsupa-orig}
  \inf_{w\in \mathrm{Ker}\mathscr{B}} \sup_{v\in \mathrm{Ker}\mathscr{B}} \frac{a(w, v)}{\|w\| \|v\|} \geq \alpha >0
\end{equation}
 For some mixed system, we can prove the so called ``ellipticity'' condition
\begin{equation} \label{ellip}
 \inf_{v\in \mathrm{Ker}\mathscr{B}} \frac{a(v, v)}{\|v\|^2} \geq \alpha >0,
\end{equation}
which is a stronger condition than the inf-sup condition (\ref{infsupa-orig}). 
Thus the ellipticity condition (\ref{ellip}) together with (\ref{infsupb-orig}),
will give us a sufficient condition for 
the well-posedness of the saddle point system (\ref{1foldmixed1-orig})-(\ref{1foldmixed2-orig}).
\end{remark}

\begin{theorem}
  The inf-sup condition for $b(\tilde{q}, \tilde\bv) =b_1(q, \bv)+b_2(\mu, \bm)$ is satisfied if and only if the corresponding inf-sup conditions 
 for $b_1(q, \bv)$ and $b_2(\mu, \bm)$ are satisfied.
\end{theorem}
\begin{proof}
 First if the inf-sup condition for $b(\tilde{q}, \tilde\bv)$ is satisfied, 
then by Theorem \ref{Thm-LBB}, we know that the operator
\begin{equation*}
 \begin{split}
  \mathscr{B}: & \mathbb{V}\times\mathbb{M} \rightarrow \mathbb{P}'\times\Lambda' \\
     & (\bv, \bm)\mapsto \mathscr{B}(\bv, \bm)
     \text{ such that } \langle \mathscr{B}(\bv, \bm),(q,\mu) \rangle
    = b_1(q,\bv)+b_2(\mu, \bm)
     \qquad \forall (q,\mu)\in \mathbb{P}\times\Lambda
 \end{split}
\end{equation*}
is onto. Therefore, the operator
\begin{equation*}
 \begin{split}
  \mathscr{B}_1: & {\mathbb{V}} \rightarrow {\mathbb{P}}' \\
     & \bv\mapsto \mathscr{B}_1 \bv
     \text{ such that } \langle \mathscr{B}_1 \bv,{q}\rangle = b_1({q}, \bv)
     \qquad \forall {q}\in {\mathbb{P}}
 \end{split}
\end{equation*}
and the operator
\begin{equation*}
 \begin{split}
  \mathscr{B}_2: & {\mathbb{M}} \rightarrow {\Lambda}' \\
     & \bm\mapsto \mathscr{B}_2 \bm
     \text{ such that } \langle \mathscr{B}_2 \bm,{\mu}\rangle = b_2({\mu}, \bm)
     \qquad \forall {\mu}\in {\Lambda}
 \end{split}
\end{equation*}
are both onto. Thus by Theorem \ref{Thm-LBB}, the inf-sup conditions for 
$b_1(q, \bv)$ and $b_2(\mu, \bm)$ are satisfied.

Conversely, if the inf-sup conditions $b_1(q, \bv)$ and $b_2(\mu, \bm)$ are satisfied,
then the operators $\mathscr{B}_1$ and $\mathscr{B}_2$ are both onto.
Thus the operator $\mathscr{B}$ is also onto. 
Therefore we have the inf-sup condition for $b(\tilde{q}, \tilde\bv) =b_1(q, \bv)+b_2(\mu, \bm)$.
\end{proof}

Therefore, to verify the inf-sup condition for $b(\tilde{q}, \tilde\bv) =b_1(q, \bv)+b_2(\mu, \bm)$,
it's enough to verify the inf-sup conditions for $b_1(q, \bv)$ and $b_2(\mu, \bm)$ individually.
Notice that the bilinear form $b_1(q, \bv)$ in (\ref{b1}) is exactly the one in the incompressible elasticity
\cite{LeTallec},
while $b_2(\mu, \bm)$ in (\ref{b2}) is exactly the one in the harmonic map problem \cite{hu2009saddle}.

For the inf-sup condition for $b_1(q, \bv)$, it's well-known that it's satisfied at the strain-free state,
where $F=I$, and the inf-sup condition is reduced to the one in the Stokes problem:
 \begin{equation} \label{infsupb1-init}
  \inf_{q \in L^2(\Omega)} \sup_{\bv \in {\mathbf H}^1_{0|\Gamma_0}(\Omega)}
  \frac{\langle q, \mathrm{div}(\bv) \rangle}{\|q\|_0 \|\bv\|_1} \geq \beta_1 > 0
 \end{equation}
Since the stress-free state has constant $F$ matrix, by change of variables, it's easy to verify that the inf-sup condition
for $b_1(q, \bv)$ is satisfied at the \textit{stress-free} state, as well.
In the general case that $\bu \neq 0$ and $F$ is not a constant, 
analytical verification of the inf-sup condition for $b_1(q,\bv)$ can be very difficult.

As for the inf-sup condition for $b_2(\mu, \bm)$, a slight modification of the proof in \cite{hu2009saddle} gives us
the following theorem:
\begin{theorem}
 Assume $\bn\in {\mathbf H}^1_{\bn_0|\Gamma_1}(\Omega) \bigcap W^{1,\infty}(\Omega)$,
 then the second inf-sup condition for $b_2(\mu, \bm)$ is satisfied, that is
 \begin{equation} \label{infsupb2-con}
   \inf_{\mu \in H^{-1}_{\Gamma_1}(\Omega)} \sup_{\bm\in {\mathbf H}^1_{{\mathbf 0}|\Gamma_1}(\Omega)}
   \frac{\langle 2\bn\cdot\bm, \mu \rangle}{\|\bm\|_1 \|\mu\|_{-1}}
   \geq \beta_2 > 0.
 \end{equation}
\end{theorem}

Finally, for the ellipticity condition for the bilinear form $a(\tilde\bw, \tilde\bv)$,
it's generally very complicated to verify due to the complexity of the expressions of $a_1(\cdot,\cdot)$, 
$a_2(\cdot,\cdot)$ and $a_3(\cdot,\cdot)$. However, it can be verified that at the stress-free state,
the ellipticity condition for $a(\tilde\bw, \tilde\bv)$ is \textit{not} satisfied. 
This doesn't necessarily mean that the linearized system (\ref{mixed1})-(\ref{mixed4}) is not well-posed, though.
After all, as remarked before, the ellipticity condition is a sufficient condition, not a necessary condition.

\section{Discretization}\label{sec:discretization}
In the elastomer problem, we have the following variables to solve:
\begin{itemize}
 \item The displacement vector field $\bu$ and the pressure $p$,
 which is also the Lagrange multiplier for the incompressibility constraint
 $\det (F)-1=0$,
 \item The director vector field $\bn$ and the Lagrange multiplier $\lambda$ 
 for the unity constraint $|\bn|^2-1=0$.
\end{itemize}
The first pair $(\bu, p)$ is similar to those in the incompressible elasticity,
and we'll use the Taylor-Hood element $P_2-P_1$ for $(\bu, p)$.
That is continuous piecewise quadratic finite element for $\bu$,
and continuous piecewise linear finite element for $p$.
The Taylor-Hoold element has been proved (\cite{LeTallec}) 
to be stable at least at the strain-free state $\bu=0$.
The second pair $(\bn, \lambda)$ is similar to those in the harmonic map
problem, and we will use piecewise linear finite element for both
$\bn$ and $\lambda$, as Winther et al. did in \cite{hu2009saddle}.
Notice that it's crucial to impose Dirichlet boundary conditions for $\bn$
and $\lambda$ at the same boundary to make sure that $|\bn|=1$
is satisfied at all the mesh nodes.

Let $V_h$ denote the space of continuous piecewise linear functions
and $V_{h,g|\Gamma_0}=\{v \in V_h \cap H^1: v=g \text{ on } \Gamma_0\}$.
Let $\bV_h$ and $\bV_{h,\bg|\Gamma_0}$ be the corresponding vector version.
Let $\pi_h$ be the nodal interpolation operators onto the spaces $V_h$ and $\bV_h$.

Let $W_h$ denote the space of continuous piecewise quadratic functions
and $W_{h,g|\Gamma_0}=\{w \in W_h \cap H^1: w=g \text{ on } \Gamma_0\}$.
Let ${\mathbf W}_h$ and ${\mathbf W}_{h,\bg|\Gamma_0}$ be the corresponding vector version.

\subsection{Existence of the discrete minimization problem} 
Let 
\begin{equation}
 K_h = \{\bu_h \in \bW_{h,0|\Gamma_0}+\bu_{0h}, \int_{\Omega} q_h (\det (I+\nabla \bu_h)-1)dx=0,\forall q_h\in V_h\}
\end{equation}
\begin{equation}
 N_h = \{\bn_h \in \mathbf{V}_{h,0|\Gamma_1}+\bn_{0h}, 
\int_{\Omega} \mu_h \pi_h(|\bn_h|^2-1)dx=0,\forall \mu_h\in V_{h,0|\Gamma_1}\}
\end{equation}
Define the admissible set 
\begin{equation}
 \mathcal{A}(\bu_{0h}, \bn_{0h}) = K_h\times N_h
\end{equation}
Notice that $\bn_h \in N_h$ if and only if the function $\pi_h(|\bn_h|^2-1) \in V_{h,0|\Gamma_1}$  is identically 0,
which means $|\bn_h|=1$ at all the mesh nodes.

Let $\varphi_j, j=1,\cdots, N$ be a basis of $V_h$, and $\psi_j, j=1,\cdots,M$ be a basis of $V_{h,0|\Gamma_1}$.
And define
\begin{equation} \label{eqn:defn-gj}
 g_j (\bu_h, \bn_h) =
 \begin{cases} 
  \int_{\Omega} \varphi_j (\det (I+\nabla \bu_h)-1) dx&           1\leq j \leq N\\ 
  \int_{\Omega} \psi_{j-N} \pi_h(|\bn_h|^2-1)dx &             N+1 \leq j \leq N+M
 \end{cases}
\end{equation}
Then $g_j$ is a continuous function on 
$\left(\bW_{h,0|\Gamma_0}+\bu_{0h}\right) \times \left(\mathbf{V}_{h,0|\Gamma_1}+\bn_{0h}\right)$.
Therefore $\mathcal{A}(\bu_{0h}, \bn_{0h})$ can be written as 
the intersection of reciprocal images of $0$ of the continuous functions $g_j$.
Thus it's closed in $\left(\bW_{h,0|\Gamma_0}+\bu_{0h}\right) \times \left(\mathbf{V}_{h,0|\Gamma_1}+\bn_{0h}\right)$.

Define the energy functional
\begin{align}
 \Pi(\bu,\bn) ={}& \int_{\Omega}\left(|F|^2-(1-a)|F^T \bn|^2\right)+b|\nabla \bn|^2 \nonumber \\
              &-\int_{\Omega} \f\cdot \bu-\int_{\Gamma_2} \bg\cdot \bu da
\end{align}
where $F=I+\nabla \bu$.
Then the discrete formulation of the minimization problem is
\begin{align} \label{problem:disc-minBTWOF}
 \text{Find } (\bu_h,\bn_h) \in \mathcal{A}(\bu_{0h}, \bn_{0h}), \text{ minimizing } \Pi \text{ on } \mathcal{A}(\bu_{0h}, \bn_{0h}).
\end{align}

\begin{lemma} \label{lemma:disc-BTWcoer}
 Assume $\bn \in N_h$ and $0\leq a \leq1$, then for any matrix $F\in \mathbb{M}^{2\times 2}$, we have 
 \begin{equation}
  (|F|^2-(1-a)|F^T \bn|^2) \geq a|F|^2
 \end{equation}
\end{lemma}
\begin{proof}
 Take any point $x \in \Omega$, suppose it's in the triangle $\bigtriangleup P_1P_2P_3$. 
 Since $\bn \in N_h$, we have
 \begin{align*}
  \bn(x) ={}& \lambda_1 \bn(P_1) + \lambda_2 \bn(P_2) + \lambda_3 \bn(P_3)
 \end{align*}
 where $\lambda_i, i=1,2,3$ are barycentric coordinates. 
 As pointed out above, $\bn \in N_h$ if and only if $|\bn|=1$ at all the mesh nodes.
 Thus we have
 \begin{align*}
  |\bn(x)| 
  ={}& |\lambda_1 \bn(P_1) + \lambda_2 \bn(P_2) + \lambda_3 \bn(P_3)| \\
  \leq{}& \lambda_1 |\bn(P_1)| + \lambda_2 |\bn(P_2)| + \lambda_3 |\bn(P_3)| \\
  ={}& \lambda_1 + \lambda_2 + \lambda_3 \\
  ={}& 1
 \end{align*}
 If $|\bn(x)|=0$, then obviously the conclusion is true. 
 In the following, we assume $|\bn(x)|>0$.
 
 Let $\hat{\bn}=\bn(x)/|\bn(x)|$, then $|\hat{\bn}|=1$.
 We have
 \begin{align*}
   &(|F|^2-(1-a)|F^T \bn|^2) \\
   ={}& (|F|^2-(1-a)|\bn(x)|^2 |F^T \hat{\bn}|^2) \\
   \geq{}& (|F|^2-(1-a)|F^T \hat{\bn}|^2) \\
   \geq{}& a |F|^2
 \end{align*}
 where in the last step we have used Proposition \ref{lemma:coerBTW}.
\end{proof}
\begin{remark}
 Similar arguments work for 3D, as well.
\end{remark}

 \begin{theorem} \label{thm:disc-exist-BTWOF}
  There exists solution to the discrete minimization problem (\ref{problem:disc-minBTWOF}).
 \end{theorem}
 \begin{proof}
   Take any $(\bu_h, \bn_h)\in \mathcal{A}(\bu_{0h}, \bn_{0h})$, we have by Lemma \ref{lemma:disc-BTWcoer}
   \begin{align*}
    \Pi(\bu_h, \bn_h) 
    \geq {}& \int_{\Omega} a|(I+\nabla \bu_h)|^2 + b|\nabla \bn_h|^2 dx \\
     &-\|\f\|_{L^2(\Omega)} \|\bu_h\|_{L^2(\Omega)}-\|\bg\|_{L^2(\Gamma_2)} \|\bu_h\|_{L^2(\Gamma_2)} \\
    \geq {}&  \int_{\Omega} a|(I+\nabla \bu_h)|^2 + b|\nabla \bn_h|^2 dx \\
     & -\left(\frac{1}{\varepsilon}\|\f\|_{L^2(\Omega)}^2 + \varepsilon\|\bu_h\|_{L^2(\Omega)}^2 \right)
     -\left(\frac{1}{\varepsilon}\|\bg\|_{L^2(\Gamma_2)}^2 + \varepsilon\|\bu_h\|_{L^2(\Gamma_2)}^2 \right) \\
   \geq {}& \int_{\Omega} C_1 |(I+\nabla \bu_h)|^2 + C_2 |\nabla \bn_h|^2 dx - C_3 
   \end{align*}
   where $\varepsilon >0$ is small and $C_i>0, i=1,2,3$ are constants and 
  we have applied the generalized Poincar$\mathrm{\acute{e}}$ inequality (\cite{CI87}, p281) 
  and the Trace Theorem (\cite{evans1998c}, p258) in the last step.
   Thus $\Pi(\bu_h, \bn_h) \rightarrow \infty$ as $\|\bu_h\|_1$ or $\|\bn_h\|_1$ goes to $\infty$.
   Therefore its minimum must be achieved at a bounded subset of $\mathcal{A}(\bu_{0h}, \bn_{0h})$.
   
   On the other hand, 
   since $\mathcal{A}$ is the intersection of 
   reciprocal images of $0$ of the continuous functions $g_j$, it's a closed set.

   Now we are minimizing a continuous function $\Pi(\bu_h, \bn_h)$ 
   on a closed  bounded \textit{finite-dimensional} set, by the Weierstrass Theorem,
   we can find $(\bu_h, \bn_h) \in \mathcal{A}(\bu_{0h}, \bn_{0h})$ minimizing $\Pi$ on $\mathcal{A}(\bu_{0h}, \bn_{0h})$.
 \end{proof}
\begin{remark}
 This theorem is the discrete version of Theorem \ref{thm:cont-exist-BTWOF}. 
 Since we are dealing with finite dimensional spaces, the proof is much simpler. 
 Plus, in the continuous case we were only able to prove Theorem \ref{thm:cont-exist-BTWOF}
 for 2D, but in the discrete case, it's easy to see that the proof above applies to 3D, too.
\end{remark}

\subsection{Equilibrium equation and its linearization}
Similar to the continuous case, we start with the following energy functional with 
Lagrange multipliers
\begin{align} \label{eqn:disc-constr-energy}
 \mathcal{E}(\bu_h,\bn_h, p_h, \lambda_h) ={}& \int_{\Omega}(|F_h|^2-(1-a)|F_h^T \bn_h|^2)
                        +b|\nabla \bn_h|^2 \nonumber \\
             &-p_h(\det (F_h)-1)+\lambda(\pi_h(\bn_h, \bn_h)-1)  \\
              &-\int_{\Omega} \f\cdot \bu_h -\int_{\Gamma_2} \bg\cdot \bu_h da \nonumber
\end{align}
where $F_h = I+\nabla \bu_h$ is the deformation gradient.

Taking the first variation variation on (\ref{eqn:disc-constr-energy}) gives
us the Euler-Lagrange equations (equilibrium equations) for the constrained minimization problem (\ref{problem:disc-minBTWOF}),
and taking another variation on the Euler-Lagrange equations
gives us the linearized problems. 

Thus the equilibrium equations are, find  $(\bu_h, \bn_h, p_h, \lambda_h) \in
{\mathbf W}_{h,\bu_0|\Gamma_0} \times \bV_{h,\bn_0|\Gamma_1}
\times V_h \times V_{h,\lambda_0|\Gamma_1}$,  such that 
\begin{align}
 0 ={}&   \int_{\Omega} 2(F_h:\nabla \bv-(1-a)(F_h^T\bn_h, \nabla \bv^T \bn_h)) 
        -p_h\frac{\partial \det }{\partial F_h}:\nabla \bv \nonumber \\
       & -\int_{\Omega} \f\cdot \bv-\int_{\Gamma_2} \bg\cdot \bv da
         \label{d2EL1} \\
0 ={}& \int_{\Omega} -2(1-a)(F_h^T \bn_h, F_h^T \bm)
       +2b(\nabla \bm,\nabla \bn_h)+2\lambda_h \pi_h(\bn_h,\bm)\label{d2EL2} \\
0 ={}& \int_{\Omega}-q(\det F_h-1) \label{d2EL3}\\
0 ={}& \int_{\Omega}\mu \pi_h((\bn_h,\bn_h)-1) \label{d2EL4}
\end{align}
for any test functions $(\bv, \bm, q, \mu) \in
{\mathbf W}_{h,0|\Gamma_0} \times \bV_{h,0|\Gamma_1}
\times V_h \times V_{h,0|\Gamma_1}$, where $F_h = I+\nabla\bu_h$.

The corresponding linearized problem is, for a given $(\bu_h, \bn_h, p_h, \lambda_h) \in
{\mathbf W}_{h,\bu_0|\Gamma_0} \times \bV_{h,\bn_0|\Gamma_1}
\times V_h \times V_{h,\lambda_0|\Gamma_1}$,
find $({\bw},{\bl},o,\gamma) \in {\mathbf W}_{h,0|\Gamma_0} \times \bV_{h,0|\Gamma_1}
\times V_h \times V_{h,0|\Gamma_1}$ such that
\begin{align}
 a_1({\bw}, \bv)+a_2({\bl}, \bv)+b_1(o, \bv) ={}& L_1(\bv)
       \label{dmixed1}\\
 a_2(\bm, {\bw})+a_3({\bl}, \bm)+b_2(\gamma, \bm) ={}& L_2(\bm)
       \label{dmixed2} \\
 b_1(q, {\bw})  ={}& L_3(q) \label{dmixed3}\\
 b_2(\mu, {\bl}) ={}& L_4(\mu) \label{dmixed4}
\end{align}
is true for any
$(\bv,\bm,q,\mu) \in {\mathbf W}_{h,0|\Gamma_0} \times \bV_{h,0|\Gamma_1}
\times V_h \times V_{h,0|\Gamma_1}$. Here $a_1$, $a_2$ and $b_1$ are the same as in the continuous case, while
\begin{align}
 a_3({\bl},\bm)=&\int_{\Omega} -2(1-a)(F_h^T {\bl}, F_h^T \bm)
              +2b(\nabla \bm, \nabla {\bl})
                         +2\lambda_h \pi_h({\bl}, \bm)
\end{align}
and
\begin{equation}
 b_2(\mu, \bm) = \int_{\Omega} 2\mu \pi_h(\bn_h, \bm)
\end{equation}

\subsection{Well-posedness of the linearized system}
As in the continuous case, the well-posedness of the linearized system 
can be reduced to the verification of the inf-sup conditions for 
the bilinear forms $b_1(\cdot, \cdot)$, $b_2(\cdot, \cdot)$ and $a(\cdot, \cdot)$.

Since we used the Taylor-Hood element $\boldsymbol{P}_2-P_1$ for the $(\bu,p)$ combination,
the inf-sup condition for $b_1(\cdot, \cdot)$ is at least satisfied
at the strain-free state (Proposition 6.1 of \cite{BrezziFortin1991}).
Since at the stress-free state the deformation gradient $F$ is a constant,
by change of variables, it's easy to verify that the inf-sup condition
for $b_1(\cdot, \cdot)$ is satisfied at the stress-free state, as well.

For the inf-sup condition for $b_2(\cdot, \cdot)$, we can use the results 
from \cite{hu2009saddle}. A slight modification of the proof in \cite{hu2009saddle}
gives us the following result:
\begin{theorem}
  Assume $\bn\in {\mathbf H}^1_{\bn_0|\Gamma_1}(\Omega) \bigcap \boldsymbol{W}^{1,\infty}(\Omega)$,
and $\bn_h\in \bV_{h,\bn_0|\Gamma_1}$ satisfies
 $|\bn_h| \geq C>0$ and $\|\bn_h-\pi_h\bn\|_1 \leq \gamma/|\log(h)|^{1/2}$.
Then we can find a positive constant $\beta_2$, independent of $h$, such that
\begin{equation}
 \inf_{\mu \in V_{h,0|\Gamma_1}} \sup_{\bm \in \bV_{h,0|\Gamma_1}}
 \frac{\langle \pi_h[\bn_h\cdot\bm], \mu \rangle}{\|\mu\|_{-1} \|\bm\|_1} \geq \beta_2
\end{equation}
\end{theorem}

In general, analytical verification of the inf-sup conditions may not be easy.
However, in the discrete case, we can relate the inf-sup values to the singular
values of certain matrices. 

For the general inf-sup value $\beta_h$ in  
\begin{equation} \label{discrete-infsup}
\beta_h = 
\inf_{q_h\in \mathbb{P}_h, ||q_h||=1} 
\left\{ \sup_{v_h\in \mathbb{V}_h, ||v_h||=1} {b(q_h,v_h)} \right\}
\end{equation}
we have the following result (\cite{BrezziFortin1991}).
\begin{theorem}
 The inf-sup value $\beta_h$ in (\ref{discrete-infsup}) is equal to the smallest singular value of the matrix $S^{-\frac{1}{2}} B T^{-\frac{1}{2}}$,
 where the matrices $S$, $T$, $B$ are defined by the following equations
 \begin{align}
 \|q_h\|^2 ={}& \bq^{T} S \bq \\
 \|v_h\|^2 ={}& \bv^{T} T \bv \\
 b(q_h,v_h) ={}& \bq^{T} B \bv
\end{align}
 and $\bq$, $\bv$ are the degrees of freedom of $q_h$ and $v_h$ respectively.
\end{theorem}
To verify the inf-sup condition or ellipticity condition for $a(\cdot, \cdot)$ 
on $\mathrm{Ker}(\mathscr{B}_h)$, we are interested in calculating the inf-sup value $\hat{\beta}_h$ in
\begin{equation} \label{discrete-infsup-a}
 \hat{\beta}_h = 
 \inf_{\|u_h\|=1} \sup_{\|v_h\|=1} a(u_h, v_h) 
  \quad \text{ on } Ker(\mathscr{B}_h) \subset \mathbb{V}_h
\end{equation}
and the ellipticity constant $\hat{\alpha}_h$ in
\begin{equation} \label{discrete-ellip}
 \hat{\alpha}_h = 
 \inf_{\|v_h\|=1} a(v_h, v_h) 
  \quad \text{ on } Ker(\mathscr{B}_h) \subset \mathbb{V}_h
\end{equation}
as well.
\begin{theorem} \label{thm:disc-ellip}
 The inf-sup value $\hat{\beta}_h$ in (\ref{discrete-infsup-a}) is equal to the smallest singular value of the matrix $A_1$,
 while the ellipticity constant $\hat{\alpha}_h$ in (\ref{discrete-ellip}) is equal to the smallest eigenvalue of the matrix $A_1$,
 where $A_1$ is the lower right $(n-m)\times(n-m)$ corner of the matrix $Q^T T^{-1/2} A T^{-1/2} Q$. Here $n$ is the dimension
 of $\mathbb{V}_h$, $m$ is the dimension of $\mathbb{P}_h$, and the matrices $T$, $A$, $B$ are defined by the following equations
\begin{align*}
 \|v_h\|^2 ={}& \bv^{T} T \bv \\
 a(u_h,v_h) ={}& \bu^{T} A \bv \\
 v_h\in \mathrm{Ker}(\mathscr{B}_h) \Leftrightarrow{}& B\bv=0 
\end{align*}
where $\bu$, $\bv$ are the degrees of freedom of $u_h$ and $v_h$ respectively.
Here $B$ is assumed to be full-rank and 
the matrix $Q$ is defined by the QR decomposition of $(BT^{-1/2})^T$
\begin{align*}
  (BT^{-1/2})^T ={}& Q
\begin{pmatrix}
  R \\
  0
\end{pmatrix}
\end{align*}
\end{theorem} 
\begin{proof}
 First, let ${\bx} = T^{\frac{1}{2}} \bu$, $\by = T^{\frac{1}{2}} \bv$
then we have
\begin{equation}
 \hat{\beta}_h = 
 \inf_{{\bx} \in \mathrm{Ker}(\tilde{B})} \sup_{\by \in \mathrm{Ker}(\tilde{B})}
\frac{{\bx}^{T} \tilde{A} \by}{\sqrt{{\bx}^{T} {\bx}}\sqrt{\by^{T} \by}}
\end{equation}
where $\tilde{B}=BT^{-1/2}$, and $\tilde{A}=T^{-1/2} A T^{-1/2}$.

Since $\tilde{B}$ is full-rank, 
the matrix $R \in \mathbb{M}^{m \times m}$ 
in the QR decomposition
\begin{equation}
  \tilde{B}^{T} = Q
\begin{pmatrix}
  R \\
  0
\end{pmatrix}
\end{equation}
is non-singular.
Let
\begin{equation}
 Q^{T} {\bx} =
\begin{pmatrix}
 {\bw}_x \\
 {\bz}_x
\end{pmatrix}
\end{equation}
where ${\bw}_x \in \mathbb{R}^{m}$ and ${\bz}_x \in \mathbb{R}^{n-m}$.
Then it's easy to verify that
\begin{align*}
  {\bx} \in Ker(\tilde{B})
  &\Leftrightarrow {\bw}_x = 0
\end{align*}
Thus there will be no constraint on ${\bz}_x$.
Therefore
\begin{equation}
 \inf_{{\bx} \in \mathrm{Ker}(\tilde{B})} \sup_{{\bx} \in \mathrm{Ker}(\tilde{B})}
 \frac{{\bx}^{T} \tilde{A}\by}{\sqrt{{\bx}^{T} {\bx}}\sqrt{\by^{T} \by}}
 = \inf_{{\bz}_x} \sup_{{\bz}_y} \frac{{\bz}_x^T A_1 {\bz}_y}
    {\sqrt{{\bz}_x^T {\bz}_x}\sqrt{{\bz}_y^T {\bz}_y}}
\end{equation}
where $A_1$ is the lower right $(n-m)\times(n-m)$ corner of the matrix $Q^T \tilde{A} Q$.
Thus $\hat{\beta}_h$ is the smallest singular value of the matrix $A_1$.

Similar arguments reduce $\hat{\alpha}_h$ to the smallest eigenvalue of
the matrix $A_1$.
\end{proof}
\begin{remark}
 The matrix $B$ is full-rank if and only if the operator $\mathscr{B}_h$
 is onto, which is true if and only if the corresponding inf-sup
 condition holds.
\end{remark}

To compute the inf-sup value for $b_2$, we need to calculate the $H^{-1}_{\Gamma_1}$ norm
for any function $v_h$ in $V_{h, 0|\Gamma_1}$. By the Riesz Representation Theorem,
we can find $\tilde{v} \in H_{0|\Gamma_1}^1$, such that $\|v_h\|_{H^{-1}_{\Gamma_1}}=\|\tilde{v}\|_{H^1}$.
The $H^1$ norm of $\tilde{v}$ can be approximated by the $H^1$ norm of $\hat{v}_h$,
which is the $L^2$ projection of $\tilde{v} \in  H_{0|\Gamma_1}^1$ into $V_{h, 0|\Gamma_1}$.
Let $\{\varphi_i, i=1,...,n\}$ be a basis of $V_{h, 0|\Gamma_1}$. We want to assemble the matrix $S$ such that
\begin{equation*}
  \|v_h\|_{H^{-1}_{\Gamma_1}} \approx \|\hat{v}_h\|_{H^1} = \bv^T S \bv 
\end{equation*}
where $\bv \in \mathbb{R}^n$ is the degree of freedom for $v_h$.
\begin{theorem}
 The matrix $S=AB^{-1}A$, where the matrices $A$ and $B$ satisfy
 \begin{align*}
  \|v_h\|_{L^2} ={}& \bv^T A \bv  \\
  \|v_h\|_{H^1} ={}& \bv^T B \bv
 \end{align*}
 for any $v_h$ in $V_{h, 0|\Gamma_1}$, and $\bv \in \mathbb{R}^n$ is the degree of freedom for $v_h$.
\end{theorem}
\begin{proof}
 Let $f: H^{-1}_{\Gamma_1} \rightarrow V_{h, 0|\Gamma_1}$ be the map taking $v_h$ to $\hat{v}_h$,
 and let $\hat{\varphi}_i = f(\varphi_i)$. It's easy to see that
 \begin{equation}
  S_{ij} = \langle \hat{\varphi}_i, \hat{\varphi}_j \rangle_{H^1}
 \end{equation}
By definition of $\hat{\varphi}_i$, we have
\begin{equation} \label{Geqn}
 \int\varphi_i\varphi_j = \int D\hat{\varphi}_i D\varphi_j + \int \hat{\varphi}_i \varphi_j  \qquad \forall 1\leq i,j\leq n
\end{equation}
Since $\hat{\varphi}_i \in V_{h, 0|\Gamma_1}$, we can write
\begin{equation*}
 \hat{\varphi}_i = \sum_k G_{ik} \varphi_k
\end{equation*}
Substituting it into (\ref{Geqn}) gives
\begin{align*}
 \int \varphi_i \varphi_j ={}& \sum_k  G_{ik} \left(\int D\varphi_k \cdot D\varphi_j +\int \varphi_k \varphi_j \right)
\end{align*}
That is $A=GB$, or $G=AB^{-1}$.
Therefore
\begin{align*}
 S ={}& (\langle \hat{\varphi}_i, \hat{\varphi}_j \rangle_{H^1})  \\
   ={}& (D\hat{\varphi}_i, D\hat{\varphi}_j)+(\hat{\varphi}_i, \hat{\varphi}_j) \\
   ={}& \sum_{p,q} G_{ip} G_{jq} \left[ (D\varphi_p, D\varphi_q)+(\varphi_p, \varphi_q) \right] \\
   ={}& \sum_{p,q} (G_{ip} B_{pq} G^T_{qj}) \\
   ={}& GBG^T \\
   ={}& (AB^{-1})B(B^{-1}A) \\
   ={}& AB^{-1}A
\end{align*}
\end{proof}

\subsection{Existence and Uniqueness of the Lagrange multipliers}
Le Tallec proved in \cite{le1981compatibility} that if the inf-sup condition is satisfied at $\bu_h$, 
then there exists a \textit{unique} $p_h$ such that $(\bu_h, p_h)$ is a solution of the discrete equilibrium
equations of the incompressible elasticity. 
In this section, we will prove similar results for our elastomer problem.

The proof uses the following result of Clarke \cite{clarke1976new} on constrained minimization problems,
where $\partial g(x)$ is the generalized gradient introduced in \cite{clarke1975generalized}.
\begin{theorem} \label{thm:Clarke1976}
 Denote $J$ a finite set of integers. We suppose given: $E$ a Banach space, $g_0$, $g_j (j\in J)$
 locally Lipschitz functions from $E$ to $\mathbb{R}$, and $C$ a closed subset of $E$. We consider
 the following problem:
 \begin{align}
  &\text{Minimize } g_0(x) \nonumber \\
  &\text{subject to } x\in C, \quad g_j(x)=0, \quad \forall j\in J. \label{eqn:ConsMin}
 \end{align}
 If $\bar{x}$ is a local solution of (\ref{eqn:ConsMin}) then there exist real numbers 
 $r_0$, $s_j$ not all zero, and a point $\xi$ in the dual space $E'$ of $E$ such that:
 \begin{equation}
  \xi \in r_0 \partial g_0(\bar{x})+\sum_{j} s_j \partial g_j(\bar{x}), \qquad -\xi\in N_c(\bar{x}),
 \end{equation}
 where $N_c(\bar{x})$ is the normal cone at $C$ in $\bar{x}$, $\partial g_j$ is the generalized
 gradient of $g_j(x)$.
\end{theorem}

\begin{theorem}
 Suppose $(\bu_h, \bn_h) \in K_h\times N_h$, and at $(\bu_h, \bn_h)$,
 the inf-sup conditions for $b_1$ and $b_2$ 
 are both satisfied. Then there exist a unique $p_h \in V_h$ and a unique $\lambda_h \in V_{h, \lambda_0|\Gamma_0}$ 
 such that $(\bu_h, \bn_h, p_h, \lambda_h)$ is a solution of the discrete equilibrium equations 
 (\ref{d2EL1})-(\ref{d2EL4}).
\end{theorem}
\begin{proof}
 Denote 
 \begin{align}
  & E = C = \left(\bW_{h,0|\Gamma_0}+\bu_{0h}\right) \times \left(\mathbf{V}_{h,0|\Gamma_1}+\bn_{0h}\right) \\
  & g_0(x) = \Pi(\bv_h, \bm_h), \qquad g_j(x)=g_j(\bv_h, \bm_h)
 \end{align}
 where the functions $g_j$ are defined as in (\ref{eqn:defn-gj}). 
 It's easy to see that
 \begin{equation}
  N_C(\bar{x}) = N_E(\bar{x}) = {(0,0)}.
 \end{equation}
 Notice that
 \begin{align} \label{eqn:dPi}
  \partial \Pi(\bu_h, \bn_h) \subset \left\{Dg_0^1 + Dg_0^2 \right\}
 \end{align}
 where $Dg_0^1$ and $Dg_0^2$ are  in 
 $\left[\left(\bW_{h,0|\Gamma_0}+\bu_{0h}\right) \times \left(\mathbf{V}_{h,0|\Gamma_1}+\bn_{0h}\right)\right]^*$.
 We have
 \begin{align*}
  Dg_0^1(\bu_h, \bn_h)\cdot (\bv_h, \bm_h) 
  ={}& 
  \begin{pmatrix}
    f_1(\bv_h) \\
    f_2(\bm_h)
  \end{pmatrix}
 \end{align*}
 where 
 \begin{align*}
  f_1(\bv) 
  ={}& 
 \int_{\Omega} 2(F_h:\nabla \bv-(1-a)(F_h^T\bn_h, \nabla \bv^T \bn_h))  
 \end{align*}
 and
 \begin{align*}
  f_2(\bm)
  ={}& 
  \int_{\Omega} -2(1-a)(F_h^T \bn_h, F_h^T \bm) 
        +2b(\nabla \bm, \nabla \bn_h)
 \end{align*}
 Also,
  \begin{align*}
  Dg_0^2(\bu_h, \bn_h)\cdot (\bv_h, \bm_h) 
  ={}& 
  \begin{pmatrix}
    -\int_{\Omega} \f\cdot \bv_h-\int_{\Gamma_2} \bg\cdot \bv_h da \\
    0
  \end{pmatrix}
 \end{align*}
 We also note that $g_j(\bv_h, \bm_h)$ is continuously differentiable on 
 $\left(\bW_{h,0|\Gamma_0}+\bu_{0h}\right) \times \left(\mathbf{V}_{h,0|\Gamma_1}+\bn_{0h}\right)$,
 and that
 \begin{align}
  &\partial g_j = \{Dg_j\}, \label{eqn:dgj}\\
  & Dg_j(\bu_h, \bn_h)\cdot (\bv_h, \bm_h)
  = 
  \begin{pmatrix}
   \int_{\Omega} -\varphi_j \frac{\partial \det }{\partial F}(I+\nabla \bu_h): \nabla \bv_h \\
   0
  \end{pmatrix} \nonumber \\
  & \qquad \qquad \text{ for } 1\leq j \leq N  \label{eqn:gj1-N}\\
  & Dg_j(\bu_h, \bn_h)\cdot (\bv_h, \bm_h)
  = 
  \begin{pmatrix}
   0 \\
   \int_{\Omega} \psi_{j-N} \pi_h[2 \bn_h\cdot \bm_h] dx 
  \end{pmatrix} \nonumber \\
  & \qquad \qquad \text{ for } N+1\leq j \leq N+M \label{eqn:gj1-M}
 \end{align}
 Therefore applying Theorem \ref{thm:Clarke1976}, we have: 

 There exists real numbers $r_0$, $s_j$ not all zero, such that
 \begin{equation} \label{eqn:dPi+dgj}
  0 \in r_0 \partial \Pi(\bu_h, \bn_h) + \sum_{j=1}^{N+M} s_j \partial g_j(\bu_h, \bn_h)
 \end{equation}
 Using (\ref{eqn:dPi}) and (\ref{eqn:dgj}), equation (\ref{eqn:dPi+dgj}) becomes 
 \begin{align} \label{eqn:DPi+Dgj}
  &r_0 \{Dg_0^1(\bu_h, \bn_h)+Dg_0^2(\bu_h, \bn_h)\} + \sum_{j=1}^{N+M} s_j Dg_j(\bu_h, \bn_h) = 0 \nonumber \\
  & \text {in } \left[\left(\bW_{h,0|\Gamma_0}+\bu_{0h}\right) \times \left(\mathbf{V}_{h,0|\Gamma_1}+\bn_{0h}\right)\right]^*.
 \end{align}
 Suppose now $r_0=0$. By linearity, and using equations (\ref{eqn:gj1-N}), (\ref{eqn:gj1-M}),
 we can then transform (\ref{eqn:DPi+Dgj}) to get 
 \begin{align}
  &\int_{\Omega} \left(\sum_{j=1}^N s_j \varphi_j \right) 
  \frac{\partial \det }{\partial F}(I+\nabla \bu_h): \nabla \bv_h dx = 0, \quad \forall \bv_h \in \bW_{h,0|\Gamma_0} 
  \label{eqn:contradictLBB1}\\
  & \int_{\Omega} \left(\sum_{j=1}^M s_{N+j} \psi_j  \right)
  \pi_h[2 \bn_h\cdot \bm_h] dx =0, \quad \forall \bm_h \in \mathbf{V}_{h,0|\Gamma_1}, \label{eqn:contradictLBB2}
 \end{align}
 Since at least one $s_j$ is nonzero, at least one of the equations (\ref{eqn:contradictLBB1}) or (\ref{eqn:contradictLBB2})
 is in contradiction with the inf-sup conditions. Thus $r_0$ cannot be zero. 
 We can then divide (\ref{eqn:DPi+Dgj}) by $r_0$ to get
 \begin{align}
  &Dg_0^1(\bu_h,\bn_h)\cdot (\bv_h, \bm_h) + 1/r_0 \sum_{j=1}^M s_j Dg_j(\bu_h, \bn_h) = -Dg_0^2(\bu_h, \bn_h)\cdot(\bv_h, \bm_h), \nonumber \\
  & \qquad \forall (\bv_h,\bm_h) 
 \in \left[\left(\bW_{h,0|\Gamma_0}+\bu_{0h}\right) \times \left(\mathbf{V}_{h,0|\Gamma_1}+\bn_{0h}\right)\right] 
 \end{align}
 That is
 \begin{align}
  &f_1(\bv_h)-\int_{\Omega} p_h \frac{\partial \det }{\partial F}(I+\nabla \bu_h): \nabla \bv_h dx  
 = \int_{\Omega} \f\cdot \bv_h + \int_{\Gamma_2} \bg\cdot \bv_h da \label{eqn:constEL1} \\
 & f_2(\bm_h)+\int_{\Omega} \lambda_h \pi_h[2 \bn_h\cdot \bm_h] dx = 0 \label{eqn:constEL2}
 \end{align}
 where we have denoted
 \begin{align}
  p_h ={}& \left(\sum_{j=1}^N s_j \varphi_j\right)/r_0  \\
  \lambda_h ={}& \left(\sum_{j=1}^M s_{N+j} \psi_j\right)/r_0 
 \end{align}
 Equations (\ref{eqn:constEL1}) and (\ref{eqn:constEL2}) are precisely (\ref{d2EL1})-(\ref{d2EL2}).
 Since we also have $(\bu_h, \bn_h) \in K_h \times N_h$, 
 we conclude that $(\bu_h, \bn_h, p_h, \lambda_h)$ is a solution of (\ref{d2EL1})-(\ref{d2EL4}).
 
 Finally, if there were two different $p_h$, their difference would violate 
 the inf-sup condition for $b_1$;
 and if there were two different $\lambda_h$, their difference would violate 
 the inf-sup condition for $b_2$.
 So we have the uniqueness of both $p_h$ and $\lambda_h$.
\end{proof}

\subsection{Implementation using FEniCS}
All the computations in this project were done using FEniCS (\cite{www:fenics}), 
which is an open source software 
for automated solution of differential equations.

We want to solve the equilibrium equations (\ref{d2EL1})-(\ref{d2EL4}).
Assume $(\bar{\bu}, \bar{\bn}, \bar{p}, \bar{\lambda})$ 
are the degrees of freedom for $(\bu_h, \bn_h, p_h, \lambda_h)$.
Replacing the test functions $(\bv, \bm, q, \mu)$ by the 
corresponding basis functions, the equilibrium equations 
(\ref{d2EL1})-(\ref{d2EL4}) can be looked as a nonlinear equation 
for $(\bar{\bu}, \bar{\bn}, \bar{p}, \bar{\lambda})$:
\begin{equation}
 \mathscr{F}(\bar{\bu}, \bar{\bn}, \bar{p}, \bar{\lambda}) = 0
\end{equation}
Since it's nonlinear, we can use Newton's method to solve. 
The iterations are, for integers $k\geq 0$,
\begin{equation} \label{eqn:elastomer-newton}
 \frac{D \mathscr{F}(\bar{\bu}^k, \bar{\bn}^k, \bar{p}^k, \bar{\lambda}^k)}
 {D(\bar{\bu}, \bar{\bn}, \bar{p}, \bar{\lambda})}
(\bigtriangleup\bar{\bu}^{k}, \bigtriangleup\bar{\bn}^{k}, 
\bigtriangleup\bar{p}^{k}, \bigtriangleup\bar{\lambda}^{k})
= - \mathscr{F}(\bar{\bu}^k, \bar{\bn}^k, \bar{p}^k, \bar{\lambda}^k)
\end{equation}
It can be verified that the matrix
\begin{equation}
 A =  \frac{D \mathscr{F}(\bar{\bu}^k, \bar{\bn}^k, \bar{p}^k, \bar{\lambda}^k)}
 {D(\bar{\bu}, \bar{\bn}, \bar{p}, \bar{\lambda})}
\end{equation}
is exactly the matrix corresponding to the left side of the linearized system
(\ref{dmixed1})-(\ref{dmixed4}) (i.e., by replacing $({\bw},{\bl},o,\gamma)$
and $(\bv, \bm, q, \mu)$ by the corresponding basis functions, 
and replacing $(\bu_h, \bn_h, p_h, \lambda_h)$ by the solution at step $k$).
Each step of (\ref{eqn:elastomer-newton}) corresponds to a linear variational problem,
and can be solved by FEniCS . Thus the whole problem can be solved. 

There is a minor issue, though, about the procedure above. 
In both the equilibrium equations (\ref{d2EL1})-(\ref{d2EL4}) 
and the linearized equations (\ref{dmixed1})-(\ref{dmixed4}),
there are terms with the interpolation operator $\pi_h$, which is \textit{not} supported
in the FEniCS form language. Thus we can't directly input those bilinear and linear forms 
in the form file for FEniCS. We solved this issue by first letting
FEniCS assemble the matrix $A$ and the right-side vector $b$ without the $\pi_h$ terms, 
and then assemble the $\pi_h$ terms ourselves and update 
the matrix $A$ and the vector $b$ accordingly. 
It turns out that the $\pi_h$ terms are not very difficult to assemble.

\section{Numerical Results} \label{sec:numerical-results}
The simulation setup is as in Figure \ref{fig:clamptwosides}. The reference domain is $[0,\mathrm{AR}]\times[0,1]$.
The clamps constraint are imposed at $X=0$ and $X=\mathrm{AR}$. 
 \begin{figure}[htbp]
 \centering
 \includegraphics[width=8cm]{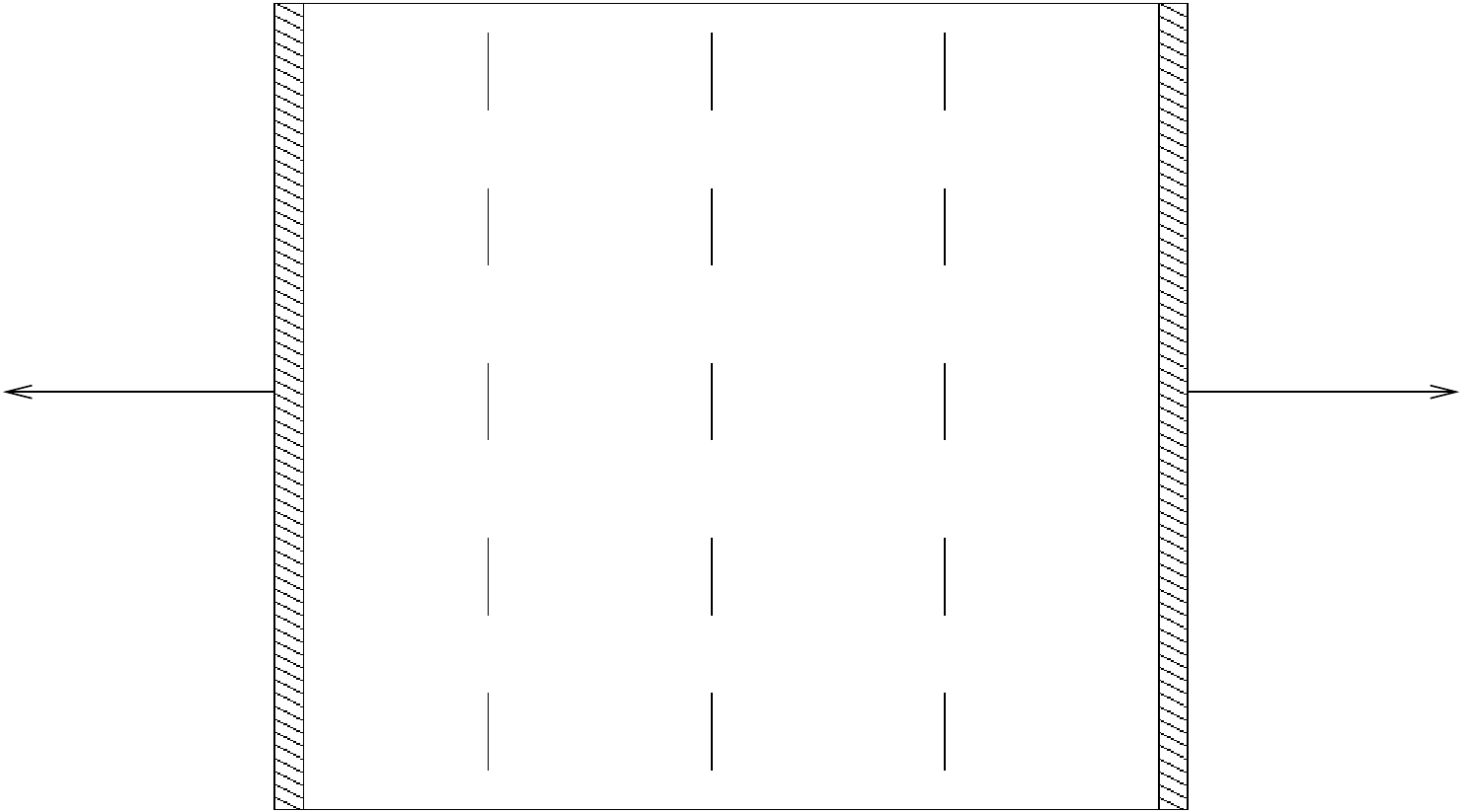}
 \caption{The elastomer is clamped and pulled on both sides.}  \label{fig:clamptwosides}
 \end{figure}
At the beginning, the elastomer is pre-stretched to the stress-free state,
and it's assumed that at this state the directors are aligned uniformly 
in the direction
\begin{equation}
 \bn = (0,1)^T
\end{equation}

The values of the displacement, pressure etc. at the stress-free state were
mentioned in section \ref{sec:stress-free}.
For completeness, we rewrite them here.
The displacement at the stress-free state is given by
\begin{align}
 u_X ={}& (a^{1/4}-1)(X-0.5\mathrm{AR}) \label{eqn:stress-free-uX} \\
 u_Y ={}& (a^{-1/4}-1)(Y-0.5) \label{eqn:stress-free-uY}
\end{align}
The pressure at the stress-free state is given by
\begin{equation} \label{eqn:stress-free-pressure}
p=2a^{1/2} 
\end{equation}
And $\lambda$ at the stress-free state is given by
\begin{equation}
\lambda = (1-a)/\sqrt{a} \label{eqn:stress-free-lambda}
\end{equation}
Suppose we want the aspect ratio at the stress-free state to be $\mathrm{AR}_n$, we need to set
\begin{equation} \label{eqn:convARnum}
 \mathrm{AR} = \frac{1}{\sqrt{a}} \mathrm{AR}_n
\end{equation}

Notice that the whole problem is symmetric about both the vertical middle
line $X=0.5\mathrm{AR}$ and the horizontal middle line $Y=0.5$.
Hence we only need to do the computation at the upper-right quarter
$[0.5\mathrm{AR}, \mathrm{AR}]\times[0.5, 1]$.
By symmetry, we have the following Dirichlet boundary conditions:
\begin{align}
 u_X = 0& \qquad \text{ on } X=0.5 \mathrm{AR} \\
 u_Y = 0&\qquad \text{ on } Y=0.5 \\
 \bn = (0,1)^T& \qquad \text{ on } X=0.5 \mathrm{AR} \text{ and } Y=0.5
\end{align}
Also, we model the ``clamped pulling'' by specifying the following
Dirichlet boundary conditions at $X=\mathrm{AR}$:
\begin{align}
 u_X ={}& 0.5\mathrm{AR}[a^{1/4}(1+Mt)-1]  \\
 u_Y ={}& (a^{-1/4}-1)(Y-0.5)  \\
 \bn ={}& (0,1)^T 
\end{align}
where $t\in[0,1]$, and $1+M$ is the largest elongation factor (ratio of the length of the elastomer
after stretching and the length of the elastomer before the stretching). 
We slowly increases $t$ from $0$ to $1$, to increase the numerical stability
of the Newton's method.
To make sure that $|\bn|=1$ at all the mesh nodes, we need to impose
Dirichlet boundary conditions for $\lambda$ at the \textit{same} boundary as $\bn$.
Therefore we impose
\begin{equation}
 \lambda = (1-a)/\sqrt{a} \quad \text{ on } X=0.5 \mathrm{AR}, X=\mathrm{AR} \text{ and } Y=0.5.
\end{equation}

We took $a=0.6$, $b=0.0015$, $\mathrm{AR}_n=1$, $M=0.4$, with ``time'' step $\bigtriangleup t = 0.01$. 
We have used uniform triangular meshes in the simulations. For example, Figure \ref{fig:mesh16-crop}
shows the mesh with mesh size $h=2^{-5}$ (notice that the mesh is $16\times 16$, but recall that this computational domain
is only $1/4$ of the true domain). 
 \begin{figure}[htbp]
 \centering
 \includegraphics[width=10cm]{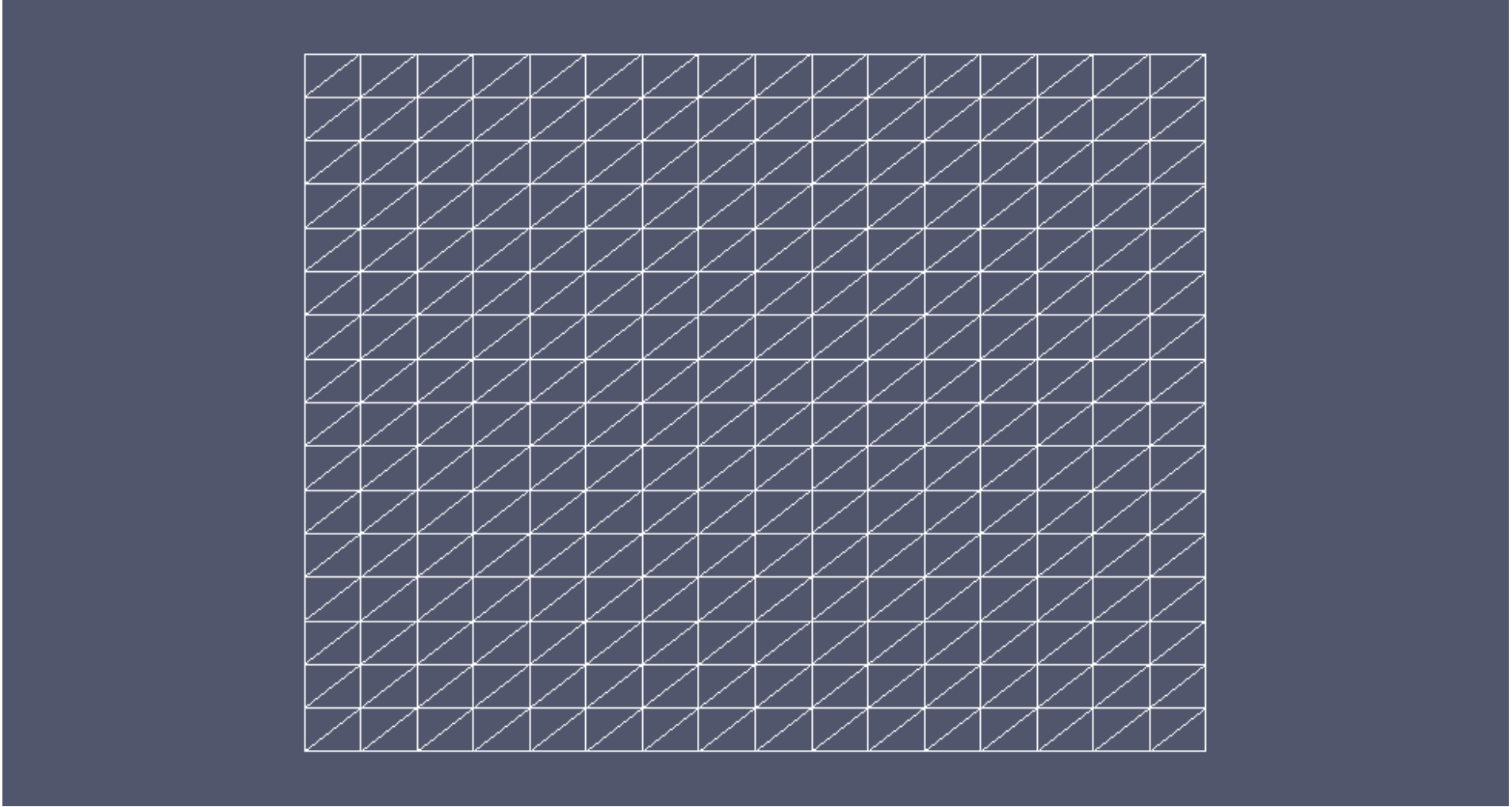}
 \caption{The uniform meshes we have used. What's shown here is the mesh with mesh size $h=2^{-5}$.}  \label{fig:mesh16-crop}
 \end{figure}

Table \ref{table:error} summarizes the numerical errors. We can see that the $L^2$ errors of $\bu_h$ and $\bn_h$,
and $H^{-1}$ error of $\lambda_h$ are relatively small, while
the $H^1$ errors of $\bu_h$ and $\bn_h$ are relatively large.
\begin{table}[htbp] 
\centering
\begin{tabular}{ | c || l | l | l | l | } 
\hline	 
$h$ & $2^{-2}$ & $2^{-3}$ & $2^{-4}$ & $2^{-5}$ \\
\hline 
$\|\bu_{h}-\bu_{h/2}\|_0$ & 3.49E-03 & 1.91E-03 & 8.39E-04 & 2.69E-04 \\ 
$\|\bu_{h}-\bu_{h/2}\|_1$ & 5.14E-02 & 3.77E-02 & 2.02E-02 & 7.66E-03 \\ 
$\|\bn_{h}-\bn_{h/2}\|_0$ & 2.32E-01 & 9.70E-02 & 3.05E-02 & 8.25E-03 \\ 
$\|\bn_{h}-\bn_{h/2}\|_1$ & 2.31E+00 & 1.91E+00 & 1.19E+00 & 6.23E-01 \\ 
$\|p_{h}-p_{h/2}\|_0$ & 1.68E-01 & 7.93E-02 & 2.38E-02 & 8.99E-03 \\ 
$\|\lambda_{h}-\lambda_{h/2}\|_{-1}$ & 1.15E-02 & 4.41E-03 & 1.51E-03 & 5.22E-04 \\ 
\hline
\end{tabular}
\caption{The numerical errors.} \label{table:error}
\end{table} 
Table \ref{table:conv-rates} summarizes the convergence rates. 
We can see that the convergence rates are relatively fast, except
the $H^1$ errors of $\bu_h$ and $\bn_h$.
\begin{table}[htbp]
\centering
\begin{tabular}{ | c || l | l | l | } 
\hline	 
$h$ & $2^{-3}$ & $2^{-4}$ & $2^{-5}$ \\
\hline 
$\log_2(\|\bu_{2h}-\bu_{h}\|_0/\|\bu_{h}-\bu_{h/2}\|_0)$ & 0.87 & 1.18 & 1.64 \\ 
$\log_2(\|\bu_{2h}-\bu_{h}\|_1/\|\bu_{h}-\bu_{h/2}\|_1)$ & 0.45 & 0.90 & 1.40 \\ 
$\log_2(\|\bn_{2h}-\bn_{h}\|_0/\|\bn_{h}-\bn_{h/2}\|_0)$ & 1.26 & 1.67 & 1.88 \\ 
$\log_2(\|\bn_{2h}-\bn_{h}\|_1/\|\bn_{h}-\bn_{h/2}\|_1)$ & 0.27 & 0.69 & 0.93 \\ 
$\log_2(\|p_{2h}-p_{h}\|_0/\|p_{h}-p_{h/2}\|_0)$ & 1.08 & 1.74 & 1.41 \\ 
$\log_2(\|{\lambda}_{2h}-{\lambda}_{h}\|_{-1}/\|{\lambda}_{h}-{\lambda}_{h/2}\|_{-1})$  & 1.38 & 1.55 & 1.54 \\ 
\hline
\end{tabular}
\caption{The convergence rates.}  \label{table:conv-rates}
\end{table}

Next, let's check the results of the inf-sup tests. Table \ref{table:infsup-t0}
lists the inf-sup values at $t=0$, where $s(\mathscr{A}|\mathrm{Ker}(\mathscr{B}))$ is the inf-sup 
value of $\mathscr{A}|\mathrm{Ker}(\mathscr{B})$, while $e(\mathscr{A}|\mathrm{Ker}(\mathscr{B}))$ is the ellipticity
constant of $\mathscr{A}|\mathrm{Ker}(\mathscr{B})$. We can see that the inf-sup values
of $b_1$ and $b_2$ don't change much with the mesh. This implies 
that the inf-sup conditions for $b_1$ and $b_2$ probably are always satisfied
regardless of the mesh. Next, we can see that the inf-sup 
values for $s(\mathscr{A}|\mathrm{Ker}(\mathscr{B}))$ are positive for all the 4 meshes, which means
the inf-sup conditions are satisfied thus
the discrete saddle point systems are well-posed for all the 4 meshes.
However, it seems that the inf-sup values $s(\mathscr{A}|\mathrm{Ker}(\mathscr{B}))$ tend to 0
as $h$ goes to 0. Finally, we can see that the ellipticity constants
$e(\mathscr{A}|\mathrm{Ker}(\mathscr{B}))$ are negative for the given 4 meshes.
This means the ellipticity conditions are not satisfied at the initial
stress-free state. This is consistent with our analytical observation.
\begin{table}[htbp]
\centering
\begin{tabular}{ | c || l | l | l | l | } 
\hline	 
$h$ & $2^{-2}$ & $2^{-3}$ & $2^{-4}$ & $2^{-5}$ \\
\hline 
$b_1$ & 0.5836 & 0.5875 & 0.5879 & 0.5880 \\ 
$b_2$ & 2.0000 & 2.0000 & 2.0000 & 2.0000 \\ 
$s(\mathscr{A}|\mathrm{Ker}(\mathscr{B}))$ & 3.60E-03 & 2.70E-04 & 5.69E-05 & 1.62E-04\\ 
$e(\mathscr{A}|\mathrm{Ker}(\mathscr{B}))$ & -3.60E-03 & -1.27E-02 & -1.78E-02 & -1.21E-02\\ 
\hline
\end{tabular}
\caption{The inf-sup values at $t=0$.}  \label{table:infsup-t0}
\end{table}
Table \ref{table:infsup-t1} lists the inf-sup values at the final state $t=1$
(elongation factor $s=1.4$). We can see that the inf-sup values for $b_1$ and $b_2$
still don't vary much with respect to the mesh size $h$. This suggests
that the inf-sup conditions for $b_1$ and $b_2$ are probably satisfied
for all the $\bu_h$ and $\bn_h$ solutions during the pulling process.
On the other hand, we still see that $s(\mathscr{A}|\mathrm{Ker}(\mathscr{B}))$ are
positive, while $e(\mathscr{A}|\mathrm{Ker}(\mathscr{B}))$ are negative, which means
the inf-sup condition for $s(\mathscr{A}|\mathrm{Ker}(\mathscr{B}))$ are satisfied,
while the ellipticity condition are not. 
Similar to $t=0$, although the inf-sup values $s(\mathscr{A}|\mathrm{Ker}(\mathscr{B}))$ are positive
for the 4 given meshes, they do seem to converge to zero
as the mesh size $h$ goes to 0.
\begin{table}[htbp]
\centering
\begin{tabular}{ | c || l | l | l | l | } 
\hline	 
$h$ & $2^{-2}$ & $2^{-3}$ & $2^{-4}$ & $2^{-5}$ \\
\hline 
$b_1$ & 0.6549 & 0.6431 & 0.6287 & 0.6163 \\ 
$b_2$ & 1.9967 & 1.9503 & 1.9065 & 1.8711 \\ 
$s(\mathscr{A}|\mathrm{Ker}(\mathscr{B}))$ & 2.91E-03 & 1.20E-03 & 5.82E-04 & 4.88E-05\\ 
$e(\mathscr{A}|\mathrm{Ker}(\mathscr{B}))$ & -2.91E-03 & -2.58E-03 & -5.82E-04 & -4.88E-05\\ 
\hline
\end{tabular}
\caption{The inf-sup values at $t=1$.}  \label{table:infsup-t1}
\end{table}

Figure \ref{fig:ssa0.6b1.5e-3} shows the graph of nominal stress vs. strain,
where we have taken the stress-free state as the reference configuration.
We can clearly see a plateau when the strain is in $(0.10, 0.22)$. Comparing with 
Figure \ref{fig:semi-soft}, we can say that we've recovered the semi-soft elasticity
phenomenon. The plateau regions in Figure \ref{fig:semi-soft} are more flat than ours,
because those graphs are for long and thin elastomers, while our graph
is for the elastomer with the aspect ratio $\mathrm{AR}_n=1$.
 \begin{figure}[htbp]
 \centering
 \includegraphics[width=8cm,angle=270]{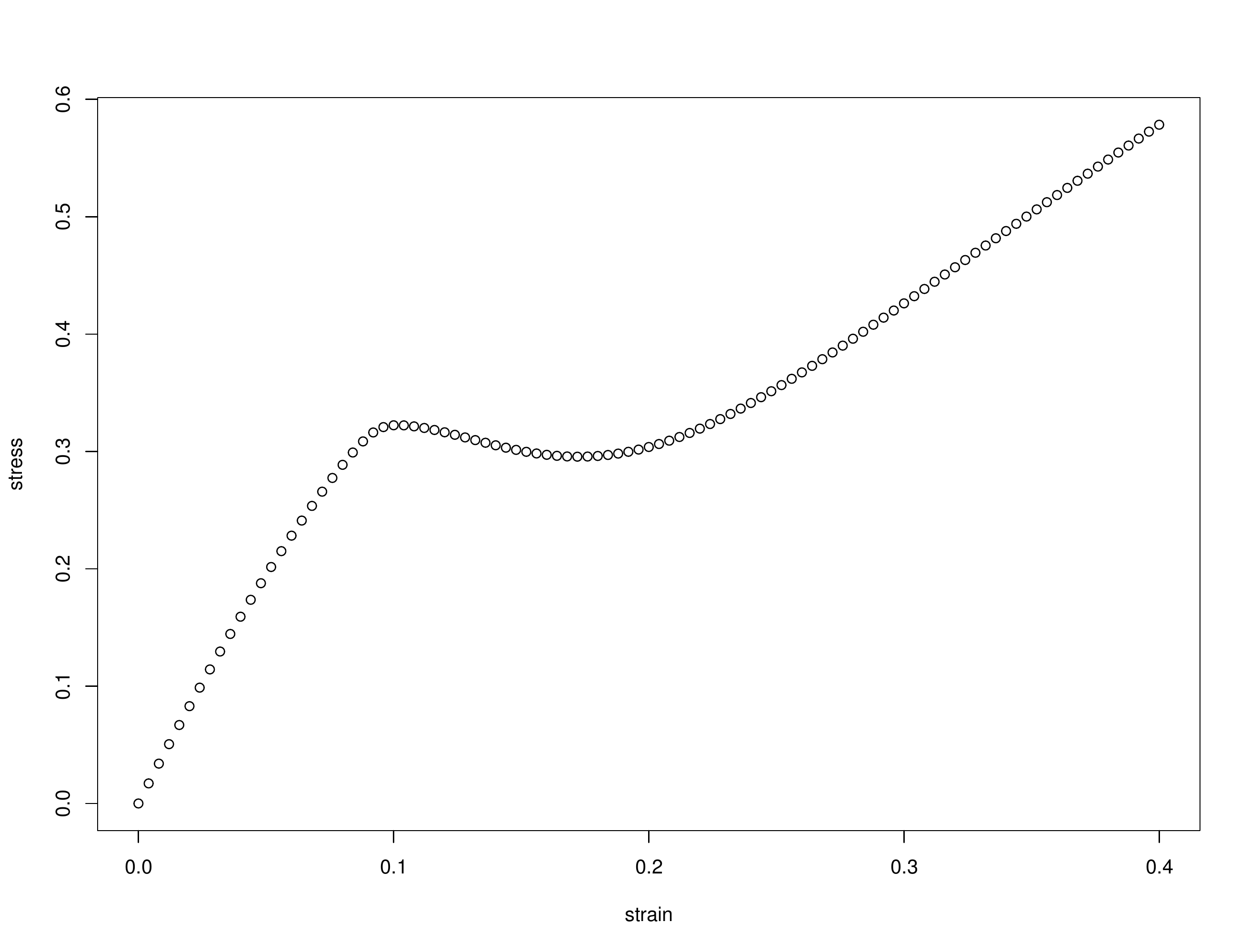}
 \caption{Nominal stress vs. strain. The plateau region corresponds to strains in the interval $(0.10, 0.22)$.}  \label{fig:ssa0.6b1.5e-3}
 \end{figure}

Figure \ref{fig:ort-platb},\ref{fig:ort-platm} and \ref{fig:ort-plate} show the directors configuration 
at the start, bottom, and end of the plateau region (elongation factor is about $1.10$, $1.17$
and $1.22$ respectively).
We can see that the directors just start to rotate at the start of 
the plateau region (Figure \ref{fig:ort-platb}), 
while many of the directors have finish the rotations at the end of the plateau region (Figure \ref{fig:ort-plate}).
Also, we can see from Figure \ref{fig:ort-platb},\ref{fig:ort-platm} and \ref{fig:ort-plate} 
that during the plateau interval $(0.10, 0.22)$,
the elastomer domain is mostly blue (low BTW energy);
while after the plateau region (Figure \ref{fig:ort-final}), the red part (hight BTW energy) starts to dominate the elastomer domain.
This means by rotating the directors, the elastomer maintains relatively low BTW energy during the elongation;
while after most of the directors have already finished the rotation, the BTW energy starts to increase.
This suggests that the plateau region in the stress-strain graph is probably caused
by the rotation of the directors. This agrees with the 
theory of liquid crystal elastomers \cite{warner2007liquid}.
 
 \begin{figure}[htbp]
 \centering
 \includegraphics[width=14cm]{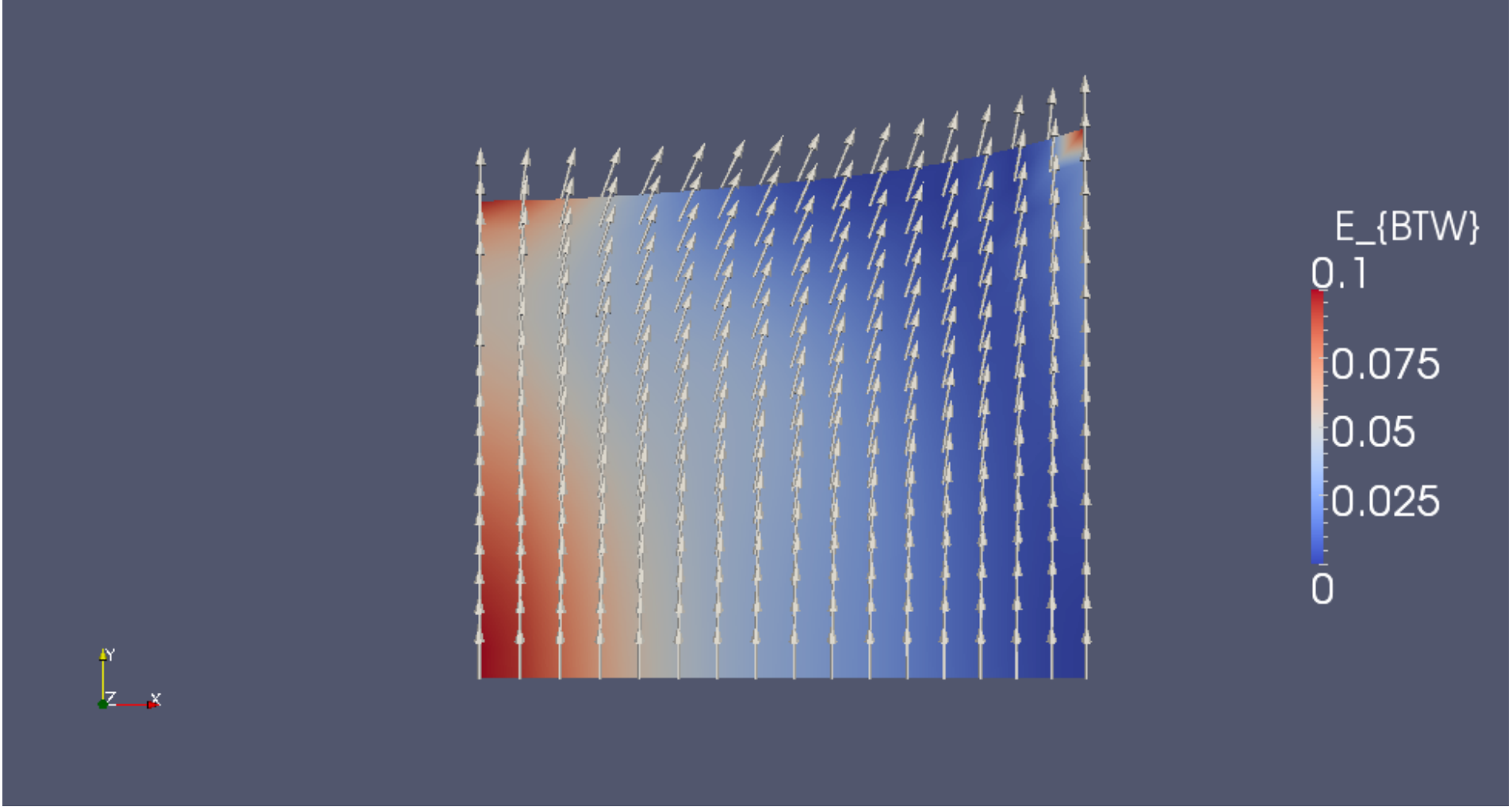}
 \caption{The directors at the beginning of the plateau region, elongation factor $s=1.10$.}  \label{fig:ort-platb}
 \end{figure}

 \begin{figure}[htbp]
 \centering
 \includegraphics[width=14cm]{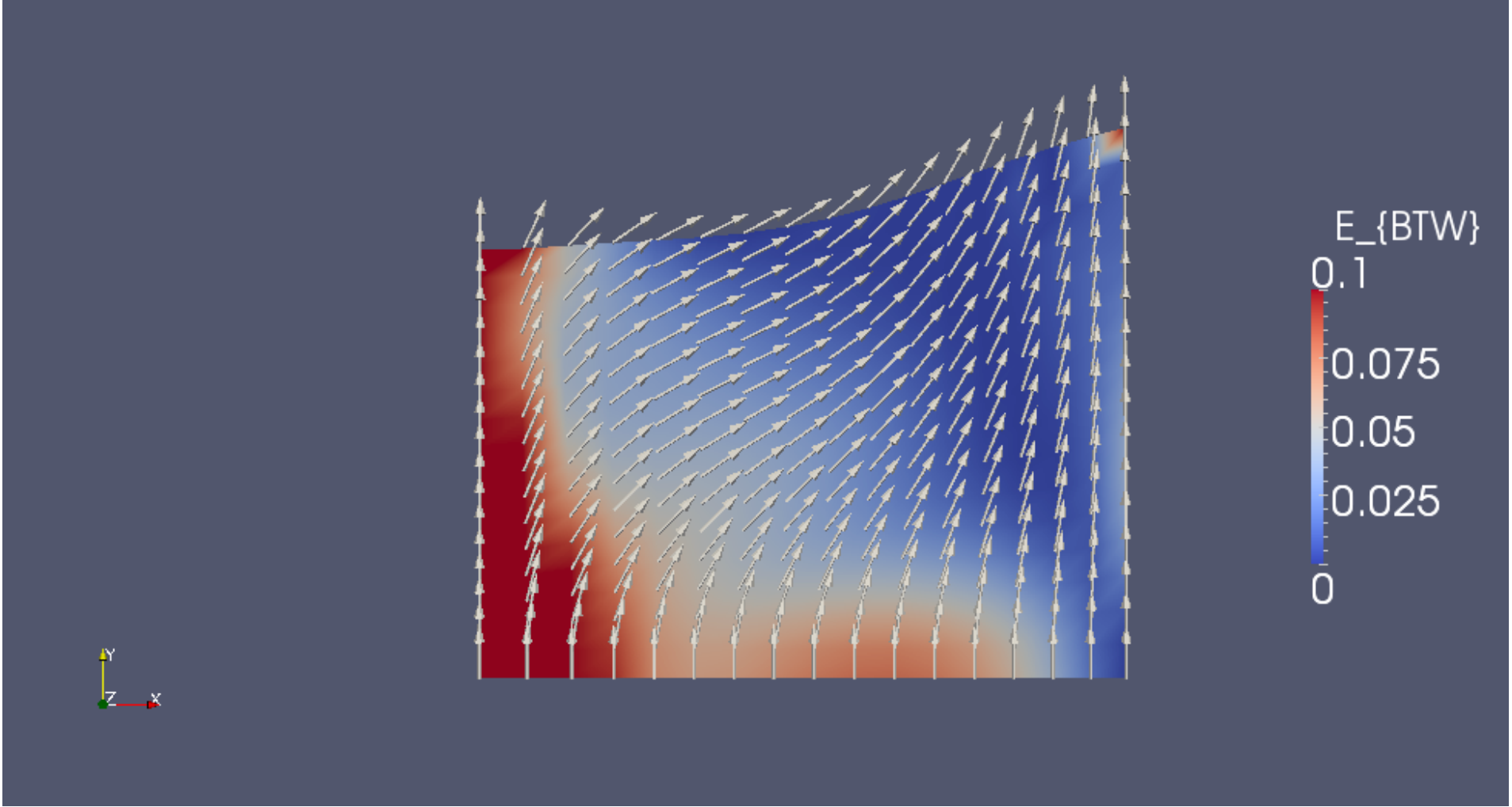}
 \caption{The directors at the bottom of the plateau region, elongation factor $s=1.17$.}  \label{fig:ort-platm}
 \end{figure}

 \begin{figure}[htbp]
 \centering
 \includegraphics[width=14cm]{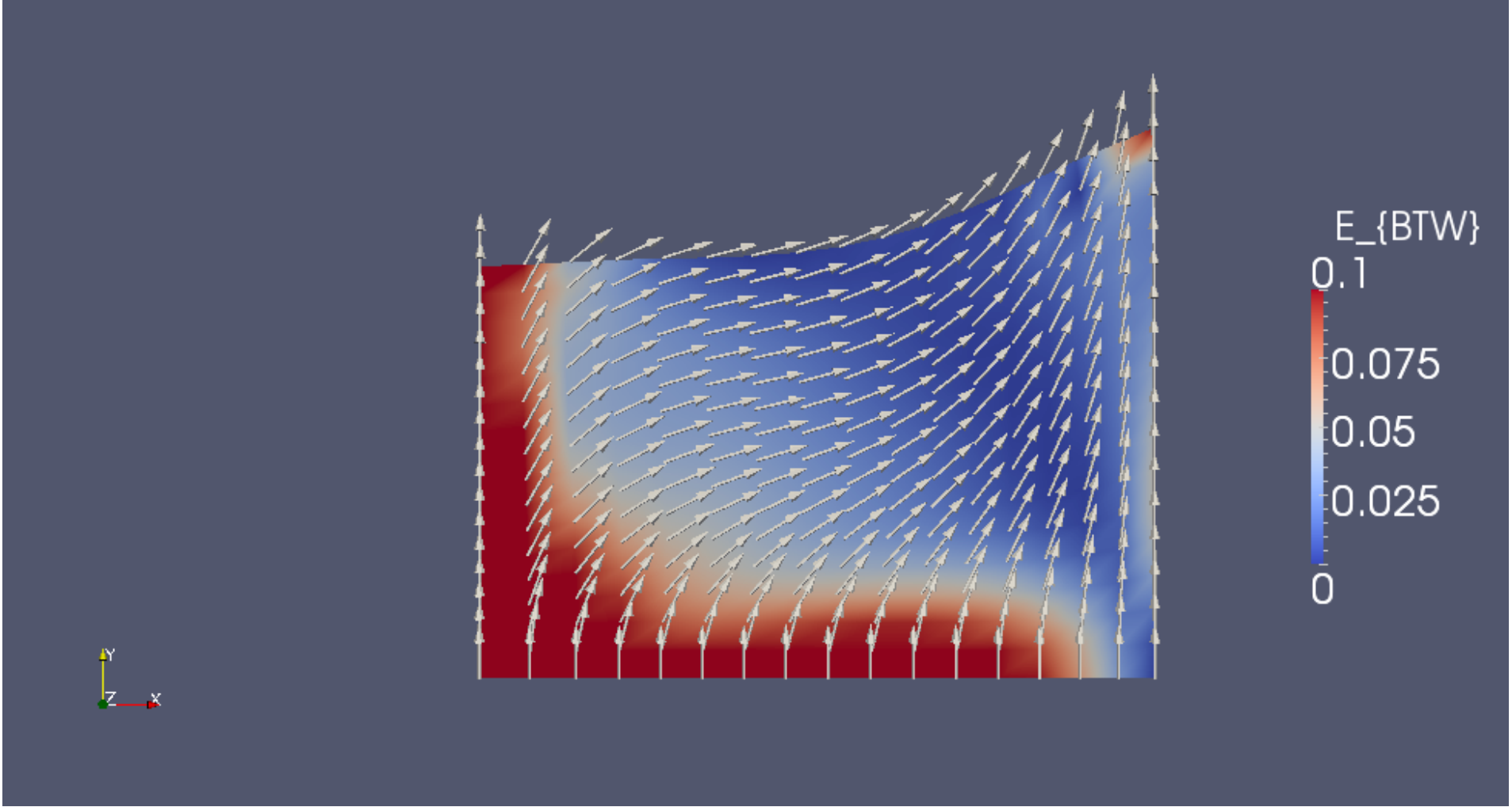}
 \caption{The directors at the end of the plateau region, elongation factor $s=1.22$.}  \label{fig:ort-plate}
 \end{figure}

Figure \ref{fig:ort-final} shows the final configuration of the directors (elongation factor $s=1.4$). 
We can see that most of the directors have finished the rotation in the final state.
 \begin{figure}[htbp]
 \centering
 \includegraphics[width=14cm]{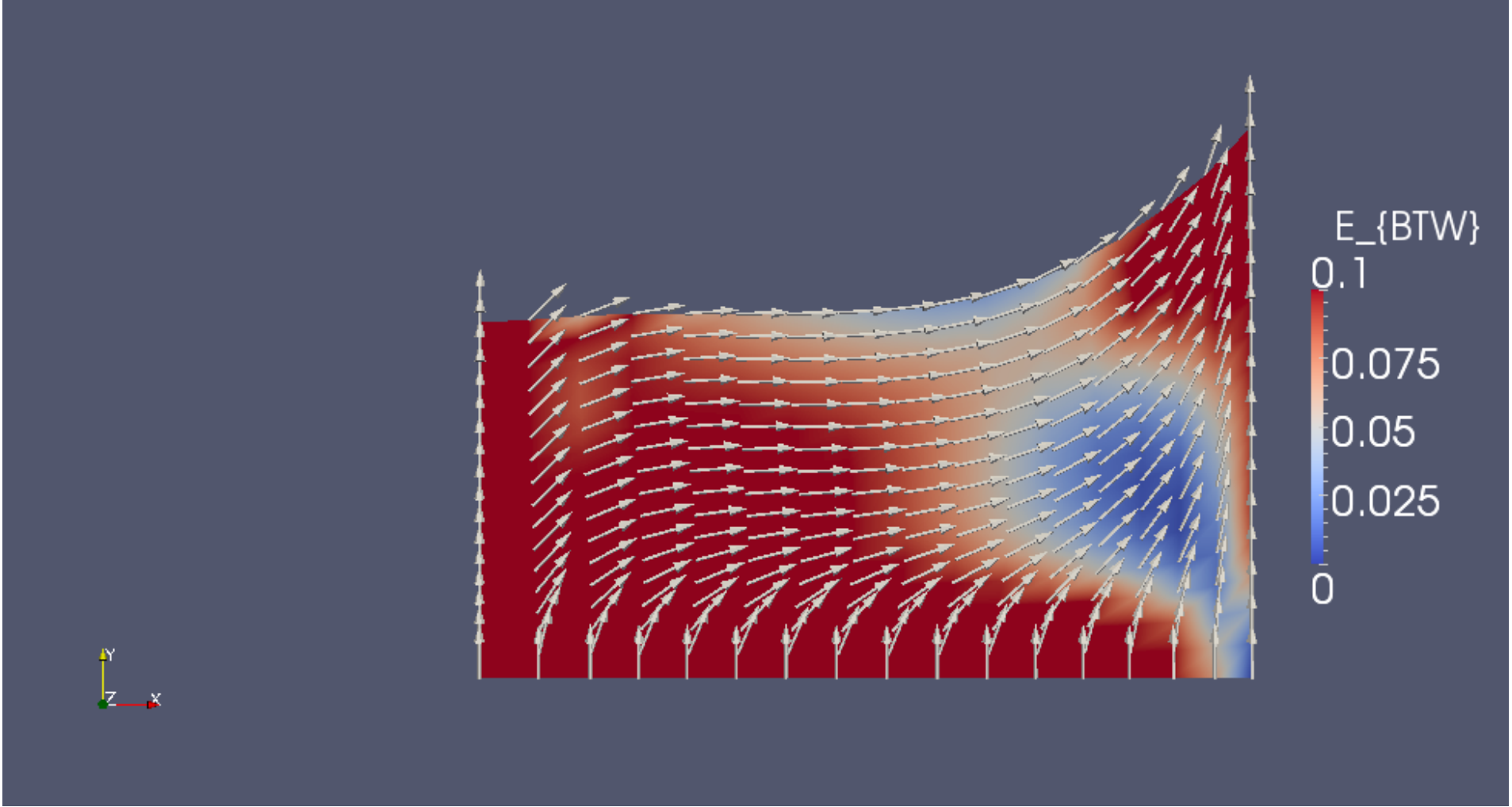}
 \caption{The directors at the final state, elongation factor $s=1.4$.}  \label{fig:ort-final}
 \end{figure}

\section{Conclusion and Discussion}\label{sec:discussion}
We have designed a numerical algorithm to simulate the 2D BTW+Oseen-Frank 
model of liquid crystal elastomer. We aimed to recover numerically
the experimentally observed phenomena such as the stripe-domain phenomenon,
the semi-soft elasticity etc. The existence, well-posedness, uniqueness
and convergence etc. of the numerical algorithm have been investigated.

We have successfully captured the semi-soft elasticity of liquid crystal
elastomer, and found it to be related with the rotation of the directors.
However, we didn't see the stripe-domain phenomenon.
We believe this is mainly because the $b$ value (the coefficient of the Oseen-Frank energy) 
$0.0015$ is relatively large
comparing to true value of $b$ in practice, which 
we estimate to be probably less than $10^{-6}$.
Since the $b$ value is relatively large, the Oseen-Frank energy
dominates the BTW energy, while the latter of which is the main driving force for the occurrence of stripe domains.
On the other hand, to observe the stripe domain phenomenon, we should use meshes
fine enough to resolve the stripes. 
For example, in the experiment of Zubarev and Finkelmann (\cite{zubarev1999monodomain}),
the width of stripe domain divided by the width of the elastomer
is around $15\mu m / 5 mm =0.003$, which is much smaller than the mesh size $2^{-5}$ of our finest mesh. 
However, very small $b$ values
require small ``time'' steps to be stable, and this together with very fine meshes
would be computationally very expensive. 
We have tried $512\times 512$ uniform mesh, and the program ran out of memory.
In the future, when the machines have much greater computing power, 
we can try very small $b$ values and very fine mesh,
and see whether the stripe domain occurs. Instead of using uniformly fine mesh,
we can also use adaptive mesh to reduce the computational cost.

Another direction worth trying is to use Ericksen or Landau-de Gennes'
model instead of the Oseen-Frank energy in the numerical simulation.
Oseen-Frank energy only allows point defects, while Ericksen or Landau-de Gennes'
model allows line and surface defects, as well (\cite{majumdar2010landau}). The stripe domains might
have line or surface defects in the transition area between the stripes, 
thus using Ericksen or Landau-de Gennes' model might have a better 
chance of capturing the stripe domain phenomenon.

\addcontentsline{toc}{section}{References}
\bibliography{numLCE}
\end{document}

%% file: my_definitions.tex
\newtheorem{theorem}{Theorem}
\newtheorem{lemma}[theorem]{Lemma}
\newtheorem{proposition}[theorem]{Proposition}
\newtheorem{corollary}[theorem]{Corollary}

\newenvironment{remark}[1][Remark]{\begin{trivlist}
\item[\hskip \labelsep {\bfseries #1}]}{\end{trivlist}}

\newcommand{\bx}{\boldsymbol{x}}
\newcommand{\by}{\boldsymbol{y}}
\newcommand{\bv}{\boldsymbol{v}}
\newcommand{\bw}{\boldsymbol{w}}
\newcommand{\bz}{\boldsymbol{z}}

\newcommand{\bV}{\mathbf{V}}
\newcommand{\bW}{\mathbf{W}}

\newcommand{\f}{\boldsymbol{f}}

\newcommand{\bn}{\boldsymbol{n}}

\newcommand{\bm}{\boldsymbol{m}}

\newcommand{\bg}{\boldsymbol{g}}

\newcommand{\tr}{{\textrm{tr}}}

\newcommand{\bq}{\boldsymbol{q}}
\newcommand{\bl}{\boldsymbol{l}}

\newcommand{\bu}{\boldsymbol{u}}